\documentclass[12pt,a4paper,twoside]{amsart}

\numberwithin{equation}{section}\usepackage{graphicx}
\usepackage{hyperref}
\usepackage{amsmath,amssymb,amsthm}
\usepackage{graphicx}

\usepackage[all]{xy}  % For commutative diagrams with \begin{CD}
\usepackage{mathrsfs}  % For script letters if needed
\usepackage{enumerate} % For custom enumerate formatting
\usepackage{rotating}
\usepackage{mathrsfs}
\usepackage{amscd,amssymb,amsopn,amsmath,amsthm,graphics,amsfonts,enumerate,verbatim,calc}
\usepackage[all]{xy}
\usepackage[utf8]{inputenc}
\usepackage{color}

\newtheorem{theorem}{Theorem}[section]
\newtheorem{claim}[theorem]{Claim}
\newtheorem{construction}[theorem]{Construction}
\newtheorem{lemma}[theorem]{Lemma}
\newtheorem{proposition}[theorem]{Proposition}
\newtheorem{fact}[theorem]{Fact}

\newtheorem{observation}[theorem]{Observation}

\theoremstyle{definition}
\newtheorem{definition}[theorem]{Definition}

\newtheorem{hypothesis}[theorem]{Hypothesis}

\theoremstyle{remark}
\newtheorem{remark}[theorem]{Remark}

\newtheorem{notation}[theorem]{Notation}

\newcommand{\End}{{\rm End}}

\newcommand{\Ext}{{\rm Ext}}

\newcommand{\otp}{{\rm otp}}

\newcommand{\supp}{{\rm supp}}

\newcommand{\Ord}{{\rm Ord}}

\newcommand{\Hom}{{\rm Hom}}

\newcommand{\bd}{{\rm bd}}

\newcommand{\vil}{\operatornamewithlimits{\varinjlim}}
\newcommand{\lo}{\longrightarrow}

\newcommand{\rest}{{\restriction}}
\newcommand{\dom}{{\rm dom}}

\newcommand{\cF}{{\mathscr F}}

\newcommand{\cT}{{\mathscr T}}

\newcommand{\cf}{{\rm cf}}

\newcount\skewfactor
\def\mathunderaccent#1#2 {\let\theaccent#1\skewfactor#2
	\mathpalette\putaccentunder}
\def\putaccentunder#1#2{\oalign{$#1#2$\crcr\hidewidth
		\vbox to.2ex{\hbox{$#1\skew\skewfactor\theaccent{}$}\vss}\hidewidth}}

\newenvironment{PROOF}[2][\proofname.]
{\begin{proof}[#1]\renewcommand{\qedsymbol}{\textsquare$_{\rm #2}$}}
	{\end{proof}}

\usepackage{hyperref}

\begin{document}

	\title[Prescribed endomorphism ring]{Representing a ring as the endomorphism ring of an abelian group}
	
	\author[M. Asgharzadeh]{Mohsen Asgharzadeh}
	\address{Mohsen Asgharzadeh, Hakimiyeh, Tehran, Iran.}
	\email{mohsenasgharzadeh@gmail.com}
	
	\author[M. Golshani]{Mohammad Golshani}
	\address{Mohammad Golshani, School of Mathematics, Institute for Research in Fundamental Sciences (IPM), P.O.\ Box:
		19395--5746, Tehran, Iran.}
	\email{golshani.m@gmail.com}
	
	\author[S. Shelah]{Saharon Shelah}
	\address{Saharon Shelah, Einstein Institute of Mathematics, The Hebrew University of Jerusalem, Jerusalem,
		91904, Israel, and Department of Mathematics, Rutgers University, New Brunswick, NJ
		08854, USA.}
	\email{shelah@math.huji.ac.il}
	
	\thanks{The second author's research has been supported by a grant from IPM (No. 1404030417). The
		third author's research was partially supported by NSF grant no. DMS 1833363. This is publication 1045b of the third author.}
	
	

 \begin{abstract}
	
We investigate Baer's realization problem for almost free abelian groups, focusing on the extent to which rings can be represented as endomorphism rings under strong freeness conditions. Building on earlier work that relied on additional set-theoretic principles such as the diamond or strong black boxes, we develop new methods that significantly weaken these assumptions.

The main result is obtained in ZFC (assuming a mild cardinal arithmetic configuration): for a strong limit singular cardinal $\mu$ with $\mu^+ < 2^\mu < 2^{\mu^+}$,  and for a wide class of cotorsion-free rings, we construct $\mu^+$-free modules whose endomorphism rings are isomorphic to the given ring. This provides a substantial partial solution to a problem of G\"{o}bel and Trlifaj.

The key innovation is the integration of $\kappa$-frame constructions with Shelah's Super Black Box, enabling a delicate diagonalization that eliminates nontrivial endomorphisms while preserving high degrees of freeness.

\end{abstract}

	\date{\today}
	\maketitle

\section{Introduction} \label{0}
The realization of rings as endomorphism rings of abelian groups, known as \emph{Baer's realization problem}, stands as a fundamental challenge at the intersection of ring theory, abelian group theory, and set theory. The central question---first posed by L\'aszl\'o Fuchs in 1958 and explicitly codified as Problem~44 in his monograph \cite{F}---asks: \begin{quote}\emph{Which rings can be realized as the endomorphism ring $\operatorname{End}_{\mathbb{Z}}(G)$ of an abelian group $G$?}\end{quote} This seemingly elementary inquiry has, over the ensuing decades, proven remarkably resilient, giving rise to a rich tapestry of partial results, counterexamples, and deep structural classifications that link the algebraic architecture of the ring to the set-theoretic and homological properties of the underlying group.

A landmark result in this direction is due to Corner \cite{Cor63}, who proved that every countable, torsion-free ring $R$ (with $1$) not containing $\mathbb{Q}$ is the endomorphism ring of a countable torsion-free abelian group. This breakthrough initiated an extensive research program, leading to deep connections between algebra and set theory.

However, Corner's construction yields groups that are not $\aleph_1$-free. A natural refinement of the problem is to demand \emph{freeness conditions} on the realizing group. Under the Generalized Continuum Hypothesis (GCH), it is possible to construct $\aleph_n$-free abelian groups with prescribed endomorphism rings \cite{579}. A major advance by G\"obel, Herden, and Shelah \cite{Sh:970} eliminated GCH, constructing $\aleph_n$-free groups for any fixed $n \in \mathbb{N}$ using the \emph{Strong Black Box}. This raised the question---posed explicitly by G\"obel and Trlifaj \cite[Problem 24.10.2]{GT}---whether such constructions can be extended to \emph{$\aleph_\omega$-free} groups using more general combinatorial principles.

Concurrently, another line of research investigated realizations under special \emph{set-theoretic hypotheses} such as the diamond principle $\diamondsuit$. For example, Dugas and G\"obel \cite{DG82} showed that if $\diamondsuit_\lambda(E)$ holds for a suitable stationary set $E$, then every $p$-cotorsion-free ring of cardinality $<\lambda$ is the endomorphism ring of a strongly $\lambda$-free $R$-module. While powerful, such results depend on axioms beyond ZFC.

The \emph{Weak Diamond Principle} $\Phi$ is a combinatorial principle concerning functions from uncountable sets to $\{0,1\}$. It is a significant weakening of Jensen's Diamond $\lozenge$, which itself follows from G\"odel's Axiom of Constructibility ($V = L$). Both principles are consistent with $\mathrm{ZFC} + \mathrm{GCH}$, with the following dependency:
\[
V = L \quad\Longrightarrow\quad \lozenge \quad\Longrightarrow\quad \Phi.
\]

To make Theorem \ref{id71} accessible, we introduce the concept of a \emph{$\kappa$-frame}---an algebraic scaffold for building modules with controlled properties. A $\kappa$-frame consists of a continuous chain of $R$-modules $G_i^*$ ($i \leq \kappa$) together with a family of test modules $H_\zeta^*$ extending $G_\kappa^*$, such that each quotient $H_\zeta^*/G_i^*$ ($i<\kappa$) is free. A \emph{nice} $\kappa$-frame adds further rigidity: explicit bases, homomorphisms $g_\zeta:G_\kappa^*\to R^2$ that cannot be extended to $H_\zeta^*$, and a non-embedding condition $H_\zeta^*/G_\kappa^*\not\hookrightarrow H_\zeta^*$. For more details, see Definition~\ref{d65}. These conditions force any endomorphism of the final constructed module to respect the stepwise structure.
 Section 3 presents our first realization theorem under weak diamond:

\begin{theorem}\label{id71}
	There exists a $\theta$-free $R$-module $M$ such that $\operatorname{End}(M,+) \cong R$, under the assumptions:
	\begin{enumerate}
		\item[(a)] $\boldsymbol{f}$ is a nice $\kappa$-frame,
		\item[(b)] $\chi < \theta \leq \lambda$,
		\item[(c)] $\bar{S}=\langle S_i: i<\lambda \rangle$ is a sequence of pairwise disjoint stationary subsets of $\{\alpha<\lambda:\cf(\alpha)=\kappa\}$
with union $S$,

         \item[(d)] weak diamond holds for each $S_i$,
		\item[(e)] $\bar\eta$ is $\theta$-free, each
$\eta_\alpha:\kappa\to\alpha$ is increasing, and the coding
condition $(e_1)$ or $(e_2)$ of Hypothesis~\ref{hp71} holds.
	\end{enumerate}
	In particular, $(M,+)$ is $\theta$-free if $(R,+)$ is $\theta$-free.
\end{theorem}

Theorem \ref{id71} establishes a general framework for realizing rings as endomorphism
rings under the Weak Diamond Principle $\Phi$. Specifically, given a nice $\kappa$-frame---an
algebraic scaffold that encodes the stepwise construction of a module---and assuming that
$\Phi$ holds on a suitable stationary set, we construct a $\theta$-free $R$-module $M$
such that $\operatorname{End}_{\mathbb{Z}}(M) \cong R$. The proof proceeds by a transfinite
induction along the frame, using weak diamond to predict and eliminate all potential
endomorphisms that are not of multiplication type. A key feature of the theorem is its
flexibility: the $\kappa$-frame formalism accommodates both countable and uncountable
cofinalities, and when the additive group of $R$ is $\theta$-free, the resulting module
$(M,+)$ inherits this freeness. This result extends earlier work under stronger
set-theoretic assumptions, such as $\diamondsuit$, and provides a unified construction
method that leverages the equivalence of weak diamond with mild cardinal arithmetic.

While Theorem \ref{id71} provides a powerful realization result under weak diamond, it naturally raises the question of whether such constructions can be carried out in ZFC alone. This is significant because weak diamond, though consistent, is not provable in ZFC---it requires mild cardinal arithmetic ($2^\lambda<2^{\lambda^+}$). The desire for ZFC theorems reflects a central theme in modern abelian group theory: to determine what can be achieved using only the standard axioms.
The Super Black Box, developed by Shelah \cite{Sh:1268}, is a combinatorial device that often eliminates set-theoretic hypotheses, yielding ZFC results under certain cardinal arithmetic configurations.
For more details, see Definition~\ref{sbbd}. Specifically, when $\mu$ is a strong limit singular cardinal and $\lambda=\min\{\sigma:2^\sigma>2^\mu\}<2^\mu$, the Super Black Box provides prediction machinery that enables a sophisticated diagonalization process. This configuration is not provable in ZFC for any cardinals, and indeed, if $\text{GCH}$ holds, it never occurs; nevertheless, we can consider the cardinal arithmetic assumptions above.

\begin{theorem}\label{id72}
	Assume that $\mu$ is a strong limit singular cardinal of cofinality
$\aleph_0$ and
\(
\Lambda=\min\{\rho:2^\rho>2^\mu\}<2^\mu.
\)
Let $R$ be a ring with $1$, $|R|<\mu$, whose additive group is
cotorsion-free and $\sigma$-free, where $\sigma$ is an infinite cardinal or
$0$.  Then there exists a $\mu^+$-free $R$-module $M$ of cardinality
$\Lambda$ such that
\(
\End_{\mathbb Z}(M)\cong R.
\)
Moreover, $(M,+)$ is $\min\{\sigma,\mu^+\}$-free, with the usual convention
that this assertion is void when $\sigma=0$.
\end{theorem}

 Theorem \ref{id72}, in contrast, dispenses with extra set-theoretic principles and
 instead relies on the Super Black Box (see Fact \ref{super-bbthm} for the existence
 of such sequences). This combinatorial device predicts, for each
 \[
 \xi \in S \subseteq \{\delta < \lambda : \operatorname{cf}(\delta) = \kappa\},
 \quad \text{with } S \text{ stationary},
 \]
 the correct local choices at coordinates $(\xi,\gamma)$, denoted   $y_{\xi,\gamma,0}$ (see
 Definition~\ref{defy}).
 The proof proceeds as follows. We begin by fixing a free $R$-module
 \(
 X = \bigoplus_{x \in \mathbf I} R x \),
 \(|\mathbf I| = \Lambda,
 \)
 which will serve as the dense free core of the final module $M$. The Super Black Box
 provides a $\mu^+$-free family of clubs $\bar{C} = \langle C_\gamma^\delta :
 \delta \in S,\ \gamma < \Lambda \rangle$ that allows us to assign local test
 modules at stationary many stages while keeping the construction free.
 For each nonzero $c \in R$ and each bit $i \in \{0,1\}$, we define a local test
 module $\Gamma_{c,i}^{\delta,\gamma}$. These modules are generated by elements
 satisfying the recursive equations
 \[
 y_n^{\delta,\gamma,c,i}
 = k_n y_{n+1}^{\delta,\gamma,c,i}
 + x_{\alpha_n}^b + i \ell_n^c b,
 \]
 where $\langle \alpha_n : n < \omega \rangle$ enumerates $C_\gamma^\delta$,
 $k_n = n+2$, and $\ell_n^c$ comes from the factorial expansion of a carefully
 chosen element of the $\mathbb Z$-adic completion.  This is the key mechanism that prevents unwanted endomorphisms from
 extending.
 The construction proceeds by building an increasing chain of approximations
 $G_f$ indexed by functions
 \[
 f : (S \cap \xi) \times \xi \longrightarrow (R \setminus \{0\}) \times 2.
 \]
 We collect all of these functions in $\mathbf P_\xi$.
 For $f \in \mathbf P_\xi$, the module $G_f$ is the $R$-submodule of the
 $\mathbb Z$-adic completion $\widehat X$ generated by $X_{<\xi}$ together
 with all local modules $\Gamma_{f(\delta,\gamma)}^{\delta,\gamma}$ for
 $(\delta,\gamma) \in \operatorname{dom}(f)$. The function $f$ thus encodes
 which test modules are attached at each stage. The Super Black Box provides
 sequences that correctly predict, for each bad endomorphism $g$, a coordinate
 $(\xi,\gamma)$ at which the attached test module permanently prevents $g$
 from extending to any later stage.
 The key technical claims \ref{pnon} and \ref{non} show that for any bad endomorphism
 $g \in \operatorname{End}(G_f, +)$ with $g(b) \notin Rb$, there exists a
 pair $(c,i) \in E$ such that attaching $\Gamma_{c,i}^{\delta,\gamma}$ at
 the next stage kills any possible extension of $g$.
 At each stationary stage $\xi \in S$, we apply the Super Black Box to choose
 the entire row of values $\{f(\xi,\gamma) : \gamma < \Lambda\}$ so that
 every bad endomorphism of $G_{f_\xi}$ is killed at some $\gamma < \Lambda$.

	\section{Conveniences}
	In this section we provide some preliminaries from algebra and set theory.
$R$ denotes an associative ring with $1$.
 By a module we mean a left $R$-module which is unitary.
	 All groups under consideration are abelian.
	For each pair of left $R$-modules $M,N$ we set $\Ext(M,N):=\Ext^1_{R}(M,N)$ and similarly $\Hom(M,N):=\Hom_{R}(M,N)$, provided the ring $R$ is clear from the context.
	The notation $(R,+)$ (resp. $(M,+)$) stands for the abelian group structure of $R$ (resp. $M$).

	\begin{definition}
		\begin{enumerate}
			\item[(i)] An abelian group $G$ is called {\em cotorsion} if $\Ext(J, G) = 0$ for all torsion-free abelian groups $J$. In other words, $G$ is cotorsion provided that it is a direct summand of every abelian group $H$ containing $G$ such that $H/G$ is torsion-free.
			\item[(ii)] An abelian group $G$ is called {\em cotorsion-free} if it does not contain any non-zero subgroup which is cotorsion.
\item[(iii)] An abelian group $G$ is called {\em  torsionless} if
for every $0\neq g\in G$ there exists
$
\phi\in\Hom_{\mathbb Z}(G,\mathbb Z)$
such that
$\phi(g)\neq0.$
		\end{enumerate}
	\end{definition}
	
	\begin{remark}
A torsionless group is cotorsion-free. It is also known (see \cite{EM02}) that an abelian group is cotorsion-free if and only if it is torsion-free and has neither a divisible non-trivial subgroup nor a subgroup isomorphic to $(J_p,+)$, the additive group of the $p$-adic numbers.
	\end{remark}
	
	\begin{definition}\label{k-free}
		An abelian group $G$ is called {\em $\lambda$-free} if every subgroup of $G$ of cardinality less than $\lambda$ is free.
	\end{definition}

	\begin{fact}
		\begin{enumerate}
			\item[(i)] Let $G$ be $\aleph_1$-free. Then $G$ is cotorsion-free.
			\item[(ii)] Let $G$ be cotorsion-free and reduced. Then $\End(G)$, as an abelian group, is cotorsion-free.
		\end{enumerate}
	\end{fact}

	\section{Weak diamond and a realization theorem}

The main result of this section is Theorem \ref{d71}. We begin by recalling the necessary preliminaries on the weak diamond principle.

\begin{definition}\label{dwd}
	Suppose $\lambda$ is an uncountable regular cardinal and $S \subseteq \lambda$ is a stationary set.
	\begin{enumerate}
		\item The \emph{diamond on $S$}, denoted $\Diamond_\lambda(S)$, asserts the existence of a sequence $\langle S_\alpha \mid \alpha \in S \rangle$ such that for each $\alpha \in S$, $S_\alpha \subseteq \alpha$, and for every $X \subseteq \lambda$, the set $\{\alpha \in S \mid X \cap \alpha = S_\alpha \}$ is stationary in $\lambda$. We write $\Diamond_\lambda$ for $\Diamond_\lambda(\lambda)$.
		
		\item The \emph{weak diamond on $S$}, denoted $\Phi_\lambda(S)$, is the statement: ``For every $c: 2^{<\lambda} \to 2$, there exists $g: S \to 2$ such that for all $f: \lambda \to 2$, the set $\{\alpha \in S \mid c(f \upharpoonright \alpha) = g(\alpha) \}$ is stationary in $\lambda$.'' We write $\Phi_\lambda$ for $\Phi_\lambda(\lambda)$.
	\end{enumerate}
\end{definition}

\begin{remark}	\begin{enumerate}
		\item
It is easily seen that $\Diamond_{\lambda^+}$ implies $2^\lambda = \lambda^+$. In fact, by a celebrated theorem of Shelah \cite{shelah}, for all uncountable cardinals $\lambda$, $\Diamond_{\lambda^+}$ is equivalent to $2^\lambda = \lambda^+$.
\item
By \cite{devlin-shelah}, $2^\lambda < 2^{\lambda^+}$ implies $\Phi_{\lambda^+}$. It was later observed by Abraham and Baumgartner that $\Phi_{\lambda^+}$ implies $2^\lambda < 2^{\lambda^+}$. Thus $\Phi_{\lambda^+}$ and $2^\lambda < 2^{\lambda^+}$ are equivalent, and consequently $\Diamond_{\lambda^+}$ implies $\Phi_{\lambda^+}$.	\end{enumerate}	
\end{remark}

	\begin{definition} \label{y37} (\cite[Definition 0.7]{Sh:1028})
		Suppose $\cF \subseteq {}^S X$ is a family of functions from $S$ into $X$, $J$ is an ideal on $S$, and $\theta$ is a cardinal.
		\begin{enumerate}
			\item We say $\cF$ is \emph{$(\theta, J)$-free} if for every $\cF' \subseteq \cF$ of cardinality $< \theta$, there is a sequence $\langle w_\eta : \eta \in \cF' \rangle$ such that:
			\begin{enumerate}
				\item $\eta \in \cF' \Rightarrow w_\eta \in J$, and
				\item if $\eta_1 \neq \eta_2 \in \cF'$ and $s \in S \setminus (w_{\eta_1} \cup w_{\eta_2})$, then $\eta_1(s) \neq \eta_2(s)$.
			\end{enumerate}
			\item We say $\cF$ is \emph{$\theta$-free} if it is $(\theta, J)$-free where $S \subseteq \Ord$ and $J = J^{\bd}_S$ is the ideal of bounded subsets of $S$.
		\end{enumerate}
	\end{definition}

	\begin{definition}\label{d65}
		Let $\kappa$ be a regular cardinal.
		\begin{enumerate}
			\item We say
			\(
			\boldsymbol{f}:= (\chi_{\boldsymbol{f}}, \zeta_{\boldsymbol{f}}, R_{\boldsymbol{f}}, \langle G^\ast_i: i \leq \kappa \rangle, \langle H^\ast_\zeta: \zeta < \zeta_{\boldsymbol{f}} \rangle)
			\)
			is a $\kappa$-frame whenever:
			\begin{enumerate}
				\item[(a):]
				
				$\chi_{\boldsymbol{f}}:=\chi \geq \kappa$ and $\zeta_{\boldsymbol{f}}$ are cardinals,
				\item[(b):] $R:=R_{\boldsymbol{f}}$ is a ring of cardinality $\leq \chi$ with $1$,
				\item[(c):]
				\begin{enumerate}
                   \item[$(c_1)$] $G^\ast_i$ is an $R$-module of cardinality $\leq \chi$ for all $i \leq \kappa$,
					\item[$(c_2)$] the sequence $\langle G^\ast_i:i\leq\kappa\rangle$ is strictly increasing and continuous,
                   \item[$(c_3)$] $G^\ast_i$ and $G^\ast_j/G^\ast_i$ are free for $i<j \leq \kappa,$
				\end{enumerate}
				\item[(d):] for all $\zeta<\zeta_{\boldsymbol{f}}$, $H^\ast_\zeta$ is an $R$-module, $G^\ast_\kappa \subseteq H^\ast_\zeta$ and $H^\ast_\zeta$ is of size $\leq \chi$,
				\item[(e):] if $i<\kappa$ and $\zeta<\zeta_{\boldsymbol{f}}$, then $H^\ast_\zeta/ G^\ast_i$ is a free $R$-module, hence $G^\ast_i$ is a direct summand of $H^\ast_\zeta$. In particular,
				the following hold:		\begin{enumerate}
					\item[$(e_1)$] By the following splitting short exact sequence $$0\longrightarrow G^\ast_i\longrightarrow H^\ast_\zeta\longrightarrow H^\ast_\zeta/ G^\ast_i\longrightarrow 0,$$ the module $H^\ast_\zeta$ is free.
					\item[$(e_2)$]  The direct limit of the following directed system:
					$${\vil}_{i<\kappa}[\cdots\lo H^\ast_\zeta/ G^\ast_i  \lo H^\ast_\zeta/ G^\ast_{i+1}\lo \cdots]\cong H^\ast_\zeta/ G^\ast_\kappa, $$ is flat as an $R$-module.
				\end{enumerate}
				
			\end{enumerate}
			\item A frame $\boldsymbol{f}$ is called \emph{nice} if in addition to the first part, it is equipped with
			\[
			\big(\textbf{I}_0, \langle \textbf{I}_{i+1}: i < \kappa \rangle, \langle G^\ast_{\kappa} \stackrel{g_\zeta} \lo R \times R  : \zeta < \zeta_{\boldsymbol{f}} \rangle, \bold{S}, \langle d_\zeta: \zeta < \zeta_{\boldsymbol{f}}   \rangle \big)
			\]
			satisfying the following properties:
			\begin{enumerate}
				\item[(f)] there is no nontrivial homomorphism from $\frac{(H^\ast_\zeta,+)}{(G^\ast_\kappa,+)}$ into $(R,+)$, or equivalently into $(H^\ast_\zeta,+)$, since $H^\ast_\zeta$ is free; see $(e_1)$.
				\item[(g)] $\textbf{I}_0$ is a free basis of $G^\ast_0$ and for each $i<\kappa$, $\textbf{I}_{i+1}$ is a free basis of $G^\ast_{i+1}$ over $G^\ast_{i}$. In particular, $G^\ast_{0}$ and $G^\ast_{i+1}/G^\ast_{i}$ are free as $R$-modules.
				\item[(h)]
\begin{enumerate}
\item[$(h_1)$] $\bold{S}:=(R\times R)\setminus\{(0,0)\}$, and the elements $d_\zeta:=(s_{\zeta,1},s_{\zeta,2})\in\bold{S}$ are chosen so that $\bold{S}=\{d_\zeta:\zeta<\zeta_{\boldsymbol f}\}$.

\item[$(h_2)$] $g_\zeta$ is a homomorphism from $G^\ast_{\kappa}$ into $R{d_\zeta} \subseteq R \times R$.
\end{enumerate}
\iffalse
				\item[(i)] if $d:=d_\zeta \in \bold{S}$, then there is no nonzero homomorphism $h$ from $(H^\ast_{\zeta},+)$ into $(R,+)^2:=(R\oplus R,+)$ extending $g_\zeta$:
				\[
				\xymatrix{
					& 0 \ar[r] & G^\ast_{\kappa} \ar[r]^{\subseteq} \ar[d]_{g_\zeta} & H^\ast_{\zeta} \ar[dl]^{\nexists h} \\
					&& R\times R. &
				}
				\]
\fi
\item[(i)] if $\zeta<\zeta_{\boldsymbol f}$ and $\tau:(Rd_\zeta,+)\to (R^2,+)$ is an additive homomorphism
such that $\tau(d_\zeta)\neq0$, then $\tau\circ g_\zeta:G_\kappa^*\to R^2$ does not extend to an additive homomorphism
$H_\zeta^*\to R^2:$	\[
\xymatrix{
	& 0 \ar[r] & G^\ast_{\kappa} \ar[r]^{\subseteq} \ar[d]_{g_\zeta} & H^\ast_{\zeta} \ar[ddl]^{\nexists } \\
	&& Rd_\zeta\ar[d]_{\tau}\\
	&& R\times R&
}
\]
	\end{enumerate}
		\end{enumerate}
	\end{definition}
We now present some cardinals $\kappa$ for which a nice $\kappa$-frame exists.

\begin{proposition}\label{d68plus}
Assume that $R$ is a ring and that $(R,+)$ is a cotorsion-free
abelian group. Put
$
\chi=|R|+\aleph_0.
$
Then there exists a nice $\aleph_0$-frame
\[
\boldsymbol f=
\bigl(
\chi,\zeta_{\boldsymbol f},R,
\langle G_i^*:i\leq\omega\rangle,
\langle H_\zeta^*:\zeta<\zeta_{\boldsymbol f}\rangle
\bigr).
\]
Moreover, every module appearing in the frame has cardinality at
most $\chi$.
\end{proposition}

\begin{proof}
Put
$
A=(R,+)^2.
$
Since $(R,+)$ is cotorsion-free, so is $A$.
We first prepare, simultaneously for every nonzero element
$b\in A$, a countable test system.
For each
$
0\neq b\in A,
$
choose
$
a_b\in\widehat{\mathbb Z}
$
such that
$
a_b b\notin A,
$
where $A$ is regarded as a subgroup of its $\mathbb Z$-adic
completion.
Such an $a_b$ exists. Indeed, if
$
a_b b\in A$
for every $a\in\widehat{\mathbb Z},
$
then $\widehat{\mathbb Z}b$ would be a nonzero cotorsion subgroup
of $A$, contradicting the fact that $A$ is cotorsion-free.
For $n<\omega$ put
\[
k_n=n+2,
\qquad
K_0=1,
\qquad
K_{n+1}=k_nK_n=(n+2)!.
\]
Choose the factorial expansion
$
a_b=\lim_{m<\omega}s_m^b,
$
where
\[
s_m^b=
\sum_{n<m}\ell_n^bK_n,
\qquad
0\leq\ell_n^b<k_n.
\]
Thus
\(
s_{n+1}^b=s_n^b+\ell_n^bK_n\)
\(n<\omega.
\)
\begin{enumerate}
	\item[$\bullet$] For every $0\neq b\in A$, let
	\(
	H_b=\bigoplus_{n<\omega}Ry_{b,n}
	\)
	be a free $R$-module.
\end{enumerate}
Also, for each \(n<\omega\) define
\(
x_{b,n}=y_{b,n}-k_ny_{b,n+1}\).
Set
$
F_b=\bigoplus_{n<\omega}Rx_{b,n}
$
and, for $m<\omega$,
$
F_{b,m}=\bigoplus_{n<m}Rx_{b,n}.
$
We record some elementary properties of these modules.
First, the elements
$
\langle x_{b,n}:n<\omega\rangle
$
are $R$-linearly independent. Indeed, suppose
$
\sum_{n\leq N}r_nx_{b,n}=0.
$
Then comparison of the coefficient of $y_{b,0}$ gives
$r_0=0$; comparison of the coefficient of $y_{b,1}$ then gives
$r_1=0$, and continuing inductively gives
$
r_0=\cdots=r_N=0.
$
Hence $F_b$ is free on the displayed basis.
Second, for every $m<\omega$,
$
H_b/F_{b,m}
$
is free. Indeed, modulo $F_{b,m}$ we have
$
y_{b,n}=k_ny_{b,n+1}
~(n<m),
$
so that the generators
\(
\{y_{b,0},\ldots,y_{b,m-1}\}
\)
successively be eliminated. Consequently,
\(
H_b/F_{b,m}
\cong
\bigoplus_{n\geq m}Ry_{b,n}.
\)

\begin{enumerate}
	\item[$\bullet$] For $m<\omega$ put
	$G_m^*
	=
	\bigoplus_{0\neq b\in A}F_{b,m},$ and
	\item[$\bullet$] $G_\omega^*
	:=
	\bigoplus_{0\neq b\in A}F_b.$
\end{enumerate}

Then each $G_m^*$ is free, the sequence
$
\langle G_m^*:m\leq\omega\rangle
$
is increasing and continuous, and whenever
$m<n\leq\omega$,
$
G_n^*/G_m^*
$
is free.
We next construct the test module belonging to a prescribed
nonzero
$
d\in R^2.
$
\begin{enumerate}
	\item[$\bullet$] Let $
	H_d
	=
	\bigoplus_{0\neq b\in A}H_b.
	$
Thus $G_\omega^*\subseteq H_d$ naturally. 

	\item[$\bullet$] Define an $R$-homomorphism \(
	g_d:G_\omega^*\to Rd\subseteq R^2
	\)
	on the free basis of $G_\omega^*$ by
	\[
	g_d(x_{b,n})=\ell_n^b d
	\qquad
	(0\neq b\in A,\ n<\omega).
	\]
\end{enumerate}

We verify the frame requirements.
For every $m<\omega$,
\(
H_d/G_m^*
\cong
\bigoplus_{0\neq b\in A}
\bigl(H_b/F_{b,m}\bigr),
\)
and hence
$
H_d/G_m^*
$
is a free $R$-module.
This proves the analogue of clause~(e) of
Definition~\ref{d65}.
We next verify clause~\ref{d65}(f).
For fixed $b\neq0$, the quotient $H_b/F_b$ is the direct limit
\[
H_b/F_b
\cong
\varinjlim
\left(
R\xrightarrow{k_0}R
\xrightarrow{k_1}R
\xrightarrow{k_2}\cdots
\right).
\]
Since $k_n=n+2$, this direct limit is naturally isomorphic to
$
\mathbb Q\otimes_{\mathbb Z}R.
$
Therefore
\(
H_d/G_\omega^*
\cong
\bigoplus_{0\neq b\in A}
\bigl(\mathbb Q\otimes_{\mathbb Z}R\bigr).
\)
In particular, as an abelian group,
$H_d/G_\omega^*$ is divisible.
Since $(R,+)$ is cotorsion-free, it contains no nonzero divisible
subgroup. This, in turn, implies that
\[
(\ast)  \quad\quad \Hom_{\mathbb Z}(H_d/G_\omega^*,R)=0.
\]
It follows also that there is no nonzero additive homomorphism
\(
H_d/G_\omega^*\to H_d.
\)
Indeed, $H_d$ is a free $R$-module; if such a homomorphism were
nonzero, composition with a suitable coordinate projection
$H_d\to R$ would contradict $(\ast) $.

It remains to establish the strengthened non-extension property.
Fix
$
d\in R^2\setminus\{0\},
$
and let
$
\tau:(Rd,+)\to(A,+)
$
be additive with
$
b_0:=\tau(d)\neq0.
$
Suppose, toward a contradiction, that there exists an additive
homomorphism
\(
\Phi:(H_d,+)\to(A,+)
\)
satisfying
\(
\Phi\restriction G_\omega^*
=
\tau\circ g_d:
\) \[
\begin{CD}
G_\omega^* @>\subseteq>> (H_d,+)  \\
@Vg_dVV @VV\Phi V \\
Rd  @>\tau>> (A,+)
\end{CD}
\]
We consider only the summand $H_{b_0}$ of $H_d$.
For $n<\omega$ put
\(
u_n=\Phi(y_{b_0,n})\in A.
\)
Then
\(
x_{b_0,n}=y_{b_0,n}-k_ny_{b_0,n+1}.
\)
Therefore
\(
u_n-k_nu_{n+1}
=
\Phi(x_{b_0,n}).
\)
It follows that
\[
\Phi(x_{b_0,n})
=
\tau\bigl(g_d(x_{b_0,n})\bigr)
=
\tau(\ell_n^{b_0}d).
\]
Although $\tau$ need not be $R$-linear, it is an additive group
homomorphism. Since $\ell_n^{b_0}\in\mathbb Z$, we have
\(
\tau(\ell_n^{b_0}d)
=
\ell_n^{b_0}\tau(d)
=
\ell_n^{b_0}b_0.
\)
Hence
$
u_n
=
k_nu_{n+1}+\ell_n^{b_0}b_0.
$
Iterating this, we obtain, for every $m<\omega$,
\(
u_0
=
K_m u_m+s_m^{b_0}b_0.
\)
Indeed, this follows immediately by induction from
$K_{m+1}=k_mK_m$.
Now pass to the $\mathbb Z$-adic completion of $A$.
Since
$
K_m=(m+1)!,
$
we have
$
(K_m u_m) 
$ convergence to zero, in the $\widehat{\mathbb Z}$-adic toplogy
On the other hand,
$
\lim (s_m^{b_0}b_0)=
a_{b_0}b_0.
$
Hence
$
u_0=a_{b_0}b_0
$
in the completion of $A$.
But
$
u_0\in A,
$
whereas by our choice of $a_{b_0}$,
$
a_{b_0}b_0\notin A,
$
contradiction.
Thus no such $\Phi$ exists. We have proved:
\[
\tau(d)\neq0
\quad\Longrightarrow\quad
\tau\circ g_d
\text{ has no additive extension from }
G_\omega^*
\text{ to }H_d.
\]
This is precisely property $(i)$.
Finally enumerate
$
(R\times R)\setminus\{(0,0)\}
=
\{d_\zeta:\zeta<\zeta_*\},
$
where
$
\zeta_*:=
|(R\times R)\setminus\{(0,0)\}|,
$
and set
\[
\zeta_{\boldsymbol f}:=\zeta_*,
\qquad
H_\zeta^*:=H_{d_\zeta},
\qquad
g_\zeta:=g_{d_\zeta}.
\]
Let
$
\mathbf I_0=\varnothing
$
and take the obvious free bases of
$G_{m+1}^*/G_m^*$ for the sets $\mathbf I_{m+1}$.
Then clauses (a)--(i) of Definition~\ref{d65} follow from the
construction.
It remains only to check cardinalities.
Since
$
|A|=|R|
$
and each block $H_b$ is countably generated,
\[
|H_d|,
\ |G_m^*|,
\ |G_\omega^*|
\leq
|R|+\aleph_0
=
\chi.
\]
Thus all modules appearing in the frame have cardinality at most
$\chi$.
Therefore $\boldsymbol f$ is a nice $\aleph_0$-frame.
\end{proof}

We now establish the same result for the case $\kappa=\aleph_1$. This time we strengthen our hypothesis a bit and require $(R,+)$ to be torsionless.
We will need the following fact.
\begin{lemma}
\label{groups1}
There is
  an increasing and continuous chain $
  \langle C_i, ~i\leq\omega_1\rangle
  $ of free abelian groups, a free abelian
group $D$ containing $C_{\omega_1}$, and a homomorphism
$
e:C_{\omega_1}\to\mathbb Z
$
such that:
\begin{enumerate}
\item[(a)] $C_j/C_i$ is free for $i<j\leq\omega_1$;
\item[(b)] $D/C_i$ is free for every $i<\omega_1$;
\item[(c)]
$
\Hom_{\mathbb Z}(D/C_{\omega_1},\mathbb Z)=0;
$
\item[(d)] for every $0\neq m\in\mathbb Z$, the homomorphism
$
me:C_{\omega_1}\to\mathbb Z
$
does not extend to a homomorphism $D\to\mathbb Z$.
\end{enumerate}
Moreover, all the groups in this test may be chosen of rank at most
$\aleph_1$.
\end{lemma}
\begin{proof}
  For each
$0\neq m\in\mathbb Z$, take, for every \(i\leq\omega_1\), free abelian groups
\(
C_i^{(m)}\)
and
\(D_m\supseteq C_{\omega_1}^{(m)},
\)
with a homomorphism
\(
h_m:C_{\omega_1}^{(m)}\to\mathbb Z
\)
satisfying the analogues of (a)--(c), such that $h_m$ does not
extend to $D_m$. We may normalize things so that
\(
h_m(C_{\omega_1}^{(m)})\subseteq m\mathbb Z.
\)
The existence of such groups is a well-known fact.
Thus
\(
\varepsilon_m:=\frac{1}{m}h_m:
C_{\omega_1}^{(m)}\to\mathbb Z
\)
is a well-defined homomorphism.  Put
\(
C_i:=\bigoplus_{m\neq0}C_i^{(m)}\)
and
\(D:=\bigoplus_{m\neq0}D_m,
\)
and define
$
e\restriction C_{\omega_1}^{(m)}=\varepsilon_m.
$
If $me$ extended to $D$, then its restriction to the $m$-th summand
$D_m$ would extend
$
m\varepsilon_m=h_m,
$
which is impossible.  The other stated properties follow by taking
direct sums.
\end{proof}
\begin{proposition}\label{d68-aleph1-plus-torsionless}
Assume that $R$ is a ring such that $(R,+)$ is torsionless.
Put
$
\chi=|R|+\aleph_1.
$
Then there exists a nice $\aleph_1$-frame
\[
\boldsymbol f=
\bigl(
\chi,\zeta_{\boldsymbol f},R,
\langle G_i^*:i\leq\omega_1\rangle,
\langle H_\zeta^*:\zeta<\zeta_{\boldsymbol f}\rangle
\bigr)
\]
such that all the modules occurring in the frame are
$\leq\aleph_1$-generated.
\end{proposition}

\begin{proof}
Fix the groups $C_i, i \leq \omega_1$, the group $D \supseteq C_{\omega_1}$, and the homomorphism
$e:C_{\omega_1}\to\mathbb Z$ as in  Fact \ref{groups1}.
Set
$
Q=D/C_{\omega_1}.
$
We now extend scalars from $\mathbb Z$ to $R$.  
\begin{enumerate}
	\item[$\bullet$]  Put
	$
	G_i^*=R\otimes_{\mathbb Z}C_i$ for
	$i\leq\omega_1$
	and
	\item[$\bullet$] $
	H^*=R\otimes_{\mathbb Z}D.
	$
\end{enumerate}

Since $(R,+)$ is torsionless, it is torsion-free and hence flat over
$\mathbb Z$.  Therefore all the relevant inclusions remain
inclusions after tensoring.  Since $C_j/C_i$ is free abelian,
\(
G_j^*/G_i^*
\cong
R\otimes_{\mathbb Z}(C_j/C_i)
\)
is a free $R$-module whenever $i<j\leq\omega_1$.  Similarly,
\(
H^*/G_i^*
\cong
R\otimes_{\mathbb Z}(D/C_i)
\)
is a free $R$-module for every $i<\omega_1$.
We next verify the analogue of clause~\textup{(f)} of
Definition~\ref{d65}.  We claim that
$
\Hom_{\mathbb Z}
\bigl(R\otimes_{\mathbb Z}Q,R\bigr)=0.
$
Suppose otherwise, and let
\(
\varphi:R\otimes_{\mathbb Z}Q\to R
\)
be a nonzero additive homomorphism.  Choose
$
x=\sum_{j=1}^n r_j\otimes q_j
$
such that $\varphi(x)\neq0$.  Since $(R,+)$ is torsionless, there
exists an additive homomorphism
$
p:R\to\mathbb Z
$
such that
$
p(\varphi(x))\neq0.
$
For each $j\leq n$, define
\(
\psi_j:Q\to\mathbb Z\)
by the assignment $q\mapsto p\bigl(\varphi(r_j\otimes q)\bigr)$.
Each $\psi_j$ is an additive homomorphism, and
\(
0\neq p(\varphi(x))
=\sum_{j=1}^n\psi_j(q_j).
\)
Thus at least one $\psi_j$ is nonzero, contradicting
$
\Hom_{\mathbb Z}(Q,\mathbb Z)=0.
$
So,
$
\Hom_{\mathbb Z}
\bigl(H^*/G_{\omega_1}^*,R\bigr)=0,
$
and since $H^*$ is a free $R$-module,  there is
no nonzero additive homomorphism from
$
H^*/G_{\omega_1}^*$ into
$ H^*$.
We now introduce the different values $d\in R^2\setminus\{0\}$.
\begin{enumerate}
	\item[$\bullet$]  For every such $d$, let $H_d^*$ be a copy of $H^*$ containing the
	common submodule $G_{\omega_1}^*$. 
	\item[$\bullet$] Define the $R$-homomorphism
	$
	g_d:G_{\omega_1}^*\to Rd\subseteq R^2
	$
	by
	$
	g_d(r\otimes x)=r\,e(x)d$
	for $r\in R,\ x\in C_{\omega_1}.
	$
\end{enumerate}

We prove the non-extension property $(i)$.  Let
$
\tau:(Rd,+)\to(R^2,+)
$
be additive and suppose
$
b:=\tau(d)\neq0.
$
Since $(R,+)$ is torsionless, so is $(R^2,+)$.  Hence there exists
an additive homomorphism
$R^2\stackrel{p_b}\longrightarrow\mathbb Z
$
such that
$
m:=p_b(b)\neq0.
$
Suppose, toward a contradiction, that there is a map
\((H_d^*,+)\stackrel{\Phi}\longrightarrow(R^2,+)
\)
making the following diagram commute:
\[
\begin{CD}
G_{\omega_1}^* @>\subseteq>> H_d^* \\
@Vg_dVV @VV\Phi V \\
(Rd,+) @>\tau>> R^2
\end{CD}
\]

Consider the natural additive embedding
\(
j:D\to H_d^*\) defined by
\(
j(y)=1_R\otimes y.
\)
Then
\(
p_b\circ\Phi\circ j:D\to\mathbb Z
\)
is an additive homomorphism.  If $x\in C_{\omega_1}$, then by
the additivity of $\tau$,
\[
\begin{aligned}
(p_b\circ\Phi\circ j)(x)
 &=p_b\bigl(\tau(g_d(1_R\otimes x))\bigr)=p_b\bigl(\tau(e(x)d)\bigr)\\
 &=p_b(e(x)b)=m e(x).
\end{aligned}
\]
Thus $p_b\circ\Phi\circ j$ extends
$
me:C_{\omega_1}\to\mathbb Z
$
to $D$, contradicting clause~\textup{(d)} of Lemma \ref{groups1}.
Therefore $\tau\circ g_d$ has no additive extension from
$G_{\omega_1}^*$ to $H_d^*$.
Finally enumerate
$
(R\times R)\setminus\{(0,0)\}
=
\{d_\zeta:\zeta<\zeta_*\},
$
put
\[
\zeta_{\boldsymbol f}:=\zeta_*,
\qquad
H_\zeta^*:=H_{d_\zeta}^*,
\qquad
g_\zeta:=g_{d_\zeta},
\]
and choose the bases $\mathbf I_0$ and $\mathbf I_{i+1}$ from the
free $R$-modules $G_0^*$ and $G_{i+1}^*/G_i^*$, respectively.
The preceding construction verifies clauses~\textup{(a)}--\textup{(i)}
of Definition~\ref{d65}.
Since $D$ and every $C_i$ have rank at most $\aleph_1$, their scalar
extensions are generated by at most $\aleph_1$ elements as
$R$-modules.  Hence all modules occurring in the frame are
$\leq\aleph_1$-generated and have cardinality at most
$
|R|+\aleph_1=\chi.
$
This completes the proof.
\end{proof}

	\begin{hypothesis}\label{hp71}
Assume that:
\begin{enumerate}
\item[(a)] $\boldsymbol f$ is a nice $\kappa$-frame and
$|R_{\boldsymbol f}|\leq\chi$;
\item[(b)] $\lambda$ is an uncountable regular cardinal and
$\chi<\theta\leq\lambda$;
\item[(c)] $\bar S=\langle S_i:i<\lambda\rangle$ is a sequence of
pairwise disjoint stationary subsets of
$\{\alpha<\lambda:\cf(\alpha)=\kappa\}$, with union $S$;
\item[(d)] $\Phi_\lambda(S_i)$ holds for every $i<\lambda$;
\item[(e)] $\bar\eta=\langle\eta_\alpha:\alpha\in S\rangle$ is
$\theta$-free, each $\eta_\alpha:\kappa\to\alpha$ is increasing,
and one of the following coding bounds holds:
\begin{enumerate}
\item[$(e_1)$] if
\(
U=\bigcup\{\operatorname{ran}(\eta_\alpha):\alpha\in S\}\)
 {and} \(\mu=|U|,
\)
then $\mu^{<\kappa}<\lambda$;
\item[$(e_2)$]
\(
\mu=
\left|\{\eta_\alpha\upharpoonright i:
\alpha\in S,\ i<\kappa\}\right|<\lambda.
\)
\end{enumerate}
\end{enumerate}
In either case the tree
\(
\mathcal T=\{\eta_\alpha\upharpoonright i:
\alpha\in S,\ i<\kappa\}
\)
has cardinality $<\lambda$.  This is the precise smallness
condition used in the weak-diamond coding below.
\end{hypothesis}

	\begin{notation}
		Assuming the above hypotheses, we set $R := R_{\boldsymbol{f}}$, and we let
	\[
	\cT := \{\eta \rest i : \eta \in \bar\eta, \; i < \kappa\}.
	\]	Let $\rho \in 	\cT$. By $\lg(\rho)$ we mean the length of $\rho$.
	\end{notation}
We now fix the initial data for the construction. This involves specifying a family of modules $\{M_\rho : \rho \in \cT\}$ indexed by the tree $\cT$, along with a compatible system of embeddings $\{h^\ast_\rho: \rho \in \cT\}$ from an increasing chain of modules $G^\ast_i$.
\begin{construction}\label{c310}
	\begin{enumerate}
	\item[(a)]	We start by constructing $M_\ast$, $\bar M_\ast$ and $\bar h^\ast$ so that the following assertions hold:
	\begin{enumerate}
		\item[$(a_1)$] $\bar M_\ast = \langle M_\rho : \rho \in \cT \rangle$,
		\item[$(a_2)$] $\bar h^\ast = \langle h^\ast_\rho : \rho \in \cT \rangle$,
		\item[$(a_3)$] $h^\ast_\rho$ is an isomorphism from $G^\ast_{\lg(\rho)}$ onto $M_\rho$, increasing with $\rho$. This property is summarized by the diagram:
		\[
		\xymatrix{
			&& G^\ast_{\lg(\rho)} \ar[d]_{\subseteq} \ar[r]^{h^\ast_\rho} & M_{\rho} \ar[d]^{\subseteq} \\
			&& G^\ast_{\lg(\varrho)} \ar[r]^{h^\ast_\varrho} & M_{\varrho},
		}
		\]
		when $\varrho$ is an extension of $\rho$.
		\item[$(a_4)$] $M_\rho \cap M_\varrho = M_{\rho \wedge \varrho}$ for each $\rho, \varrho$,
		\item[$(a_5)$] $M_\ast$ is the free sum of $\bar M_\ast$ naturally defined, i.e., $M_\ast = \bigoplus_{\rho \in \mathcal{T}} M_\rho$.
	\end{enumerate}
	
	In particular, for each $\eta \in \bar\eta$ we have the following diagram:
	\[
	\begin{CD}
	@. M_{\eta\rest 0} @>\subseteq>> M_{\eta\rest 1} @>\subseteq>> M_{\eta\rest 2} @>\subseteq>> M_{\eta\rest 3} @>\subseteq>> \ldots \\
	@. h^\ast_{\eta\rest 0} @AAA h^\ast_{\eta\rest 1} @AAA h^\ast_{\eta\rest 2} @AAA h^\ast_{\eta\rest 3} @AAA \cdots \\
	@. G^\ast_{0} @>\subseteq>> G^\ast_{1} @>\subseteq>> G^\ast_{2} @>\subseteq>> G^\ast_{3} @>\subseteq>> \ldots
	\end{CD}
	\]
For any $\eta \in \bar\eta$, we define
	\begin{itemize}
	\item[$(a_6)$]
	$M_\eta  := \bigcup_{i<\kappa} M_{\eta \rest i}, $
\item[$(a_7)$]  	$h^\ast_\eta := \bigcup_{i<\kappa} h^\ast_{\eta \rest i} : G^\ast_{\kappa} \longrightarrow M_\eta,$ where $(a_3)$ ensures that $h^\ast_\eta$ is well-defined.
	\end{itemize}
\item[(b)]	Let $\rho \in \mathcal{T}$ and suppose $\lg(\rho)$ is successor, say $\lg(\rho) = i+1$. Then the isomorphism $h^\ast_\rho: G^\ast_{i+1} \to M_\rho$ satisfies $h^\ast_\rho(G^\ast_i) = M_{\rho \upharpoonright i}$, i.e., we have the diagram
	\[
	\begin{CD}
	0 @>>> M_{\rho\rest i} @>\subseteq>> M_\rho @>>> \frac{M_\rho}{M_{\rho\rest i}} @>>> 0 \\
	@. h^\ast_\rho\rest @AAA h^\ast_\rho @AAA \exists f_i @AAA \\
	0 @>>> G^\ast_{i} @>\subseteq>> G^\ast_{i+1} @>>> \frac{G^\ast_{i+1}}{G^\ast_i} @>>> 0
	\end{CD}
	\]
	Thanks to the five lemma, the displayed diagram induces an isomorphism
	\(
	f_\rho : \frac{G^\ast_{i+1}}{G^\ast_i} \stackrel{\cong}{\longrightarrow} \frac{M_\rho}{M_{\rho\rest i}}.
	\)
	Now let $\textbf{I}_\rho$ be the image of $\textbf{I}_{i+1}$ under this isomorphism. Then $\textbf{I}_\rho$ is a free basis of $\frac{M_\rho}{M_{\rho\rest i}}$. Set also
	\begin{itemize}
		\item[$(b_1)$] $\textbf{I} := \bigcup\{\textbf{I}_\rho : \rho \in \mathcal{T} \text{ and } \lg(\rho) \text{ is not limit}\}$,
		\item[$(b_2)$] $Z := \{s_1 x_1 + s_2 x_2 : (s_1,s_2) \in (R\times R)\setminus\{(0,0)\}  \text{ and } x_1 \neq x_2 \in \textbf{I}\}$,
		\item[$(b_3)$]	$Z^\sharp:=
		\{rx+y:r\in R,\ x,y\in\boldsymbol I,\ x\neq y\}.$
	\end{itemize}
	
	\item[(c)]	
	\begin{enumerate}
		\item[$(c_1)$] For each $\eta \in \bar\eta$, $\zeta < \zeta_{\boldsymbol{f}}$, $\iota \in \{0, 1\}$ and $z \in M_\ast$ we choose $M^\ast_{\eta,\zeta,{z,\iota}}$ and $h^\ast_{\eta,\zeta,{z,\iota}}$ such that $h^\ast_{\eta,\zeta,{z,\iota}}$ is an isomorphism from $H^\ast_\zeta$ onto $M^\ast_{\eta,\zeta,{z,\iota}}$,
		
        \item[$(c_2)$] Call a pair $(\zeta,z)$ \emph{compatible} if
$d_\zeta=(s_1,s_2)$ and
\(
z=s_1x_1+s_2x_2
\)
for distinct $x_1,x_2\in\mathbf I$, in such a way that
\(
\sigma_{\zeta,z}:Rd_\zeta\to Rz,
\) defined by \(
\sigma_{\zeta,z}(a d_\zeta)=az
\)
is an $R$-isomorphism.  In particular, if
$z=rx+y\in Z^\sharp$ and $d_\zeta=(r,1)$, then
$(\zeta,z)$ is compatible: the second coordinate of $d_\zeta$
and the coefficient of $y$ in $z$ are both $1$, so both cyclic
modules are free of rank one.

        \item[$(c_3)$] For a compatible pair $(\zeta,z)$, we put
\(
g_{\zeta,z}:=\sigma_{\zeta,z}\circ g_\zeta:
G_\kappa^*\to Rz.
\)

        \item[$(c_4)$] Whenever $(\zeta,z)$ is compatible, the copies in
$(c_1)$ are chosen so that, for $\iota\in\{0,1\}$ and
$x\in G_\kappa^*$,
\(
h^*_{\eta,\zeta,z,\iota}(x)
:=h^*_\eta(x)+\iota g_{\zeta,z}(x).
\)
%Only compatible pairs are used in the diagonalization below.

        \item[$(c_5)$]
$M^*_{\eta_1,\zeta_1,z_1,\iota_1}\cap
 M^*_{\eta_2,\zeta_2,z_2,\iota_2}\subseteq M_*$ whenever
$\eta_1\neq\eta_2$, or $\zeta_1\neq\zeta_2$, or
$z_1\neq z_2$, or $\iota_1\neq\iota_2$.
	\end{enumerate}

	\item[(d)]	Let $M^\ast$ be the unique $R$-module such that
	\begin{itemize}
		\item[$(d_1)$] $M^\ast_{\eta,\zeta,{z,\iota}} \subseteq M^\ast$,
		\item[$(d_2)$] $\bigcup_{\eta,\zeta, z, \iota} M^\ast_{\eta,\zeta, z,\iota}$ generates $M^\ast$ freely over $M_\ast$.
	\end{itemize}

		\end{enumerate}
	
\end{construction}
	\begin{observation}\label{d80}
		\begin{enumerate}
\item[(i)] $M_\ast$ (see Construction~\ref{c310}$(a_5)$) is free with basis $\textbf{I}$.
			\item[(ii)] If $\Lambda \subseteq \bar\eta$ is of cardinality less than $\theta$, then $M_\Lambda := \sum_{\varrho \in \Lambda} M_\varrho \subseteq M_\ast$ is free, and $M_\ast / M_\Lambda$ is free.
		\end{enumerate}
	\end{observation}
	
	To give the next definition, let us recall that $\bar\eta = \langle \eta_\alpha: \alpha \in S \rangle$, where $S \subseteq \lambda$ is unbounded.
	With the basic building blocks $M^\ast_{\eta,\zeta,z,\iota}$ in hand, we now organize them into more complex submodules indexed by ordinals. These will allow us to systematically control the interactions between the different branches of the tree.
\begin{definition}\label{83}	\begin{enumerate}
		\item[(i)]
Let $\mathcal A$ be the set of all genuine attachment symbols
\(
\boldsymbol s=\langle\alpha,\iota,\zeta,z\rangle
\)
with $\alpha\in S$, $\iota\in\{0,1\}$,
$\zeta<\zeta_{\boldsymbol f}$, $z\in Z$, and $(\zeta,z)$ compatible
in the sense of Construction~\ref{c310}(c); for such a symbol put
\(
N_{\boldsymbol s}:=M^*_{\eta_\alpha,\zeta,z,\iota}.
\)
For each $\alpha\in S$ introduce in addition a null symbol
$\mathbf 0_\alpha$, and put $N_{\mathbf 0_\alpha}=0$.
	\item[(ii)]
For $\gamma\leq\lambda$, let $\mathcal A_\gamma$ be the set of
sequences
\(
\overline{\boldsymbol s}
 :=\langle\boldsymbol s_\beta:\beta\in S\cap\gamma\rangle
\)
such that, for every $\beta\in S\cap\gamma$, either
$\boldsymbol s_\beta=\mathbf 0_\beta$, or
$\boldsymbol s_\beta=\langle\beta,\iota,\zeta,z\rangle\in\mathcal A$.
Set
\(
M_{\overline{\boldsymbol s}}
 =M_*+\sum_{\beta\in S\cap\gamma}N_{\boldsymbol s_\beta}.
\)
	\item[(iii)]
We call $\overline{\boldsymbol s}$ from item (ii) \emph{coherent}, if for every
genuine coordinate
$\boldsymbol s_\beta=\langle\beta,\iota,\zeta,z\rangle$ there is a
finite set $u_\beta\subseteq S\cap\beta$ such that every basis
coordinate of $z$ is supported on branches
$\{\eta_\xi:\xi\in u_\beta\}$.  Equivalently, after enlarging the
finite support if necessary,
\(
z\in\sum_{\xi\in u_\beta}M_{\eta_\xi}\subseteq M_*.
\)
%Null coordinates impose no coherence requirement.	
\end{enumerate}

\end{definition}
\begin{fact}\label{unimodular}
	Every $z\in Z^\sharp$ is unimodular in $M_\ast$, in the sense that $Rz$ is a direct summand of $M_\ast$.
\end{fact}
\begin{proof} Recall from Construction~\ref{c310}~$(b_3)$
	that $Z^\sharp=
	\{rx+y:r\in R,\ x,y\in\boldsymbol I,\ x\neq y\}.$
	Suppose
	$z=rx+y \in Z^\sharp.$
	Then
	$
	y=z-rx.$
	Hence replacing $y$ by $z$ in the basis $\boldsymbol I$ gives
	again a new $R$-basis
	\(
	\bigl(\boldsymbol I\setminus\{y\}\bigr)\cup\{z\}.
	\)
	Consequently
	\(
	M_\ast=Rz\oplus
	\bigoplus_{w\in\boldsymbol I\setminus\{y\}}Rw,
	\)
	and in particular $Rz$ is a direct summand of $M_\ast$, as required.
\end{proof}

\begin{lemma}\label{d222}
Let $\overline{\boldsymbol s}\in\mathcal A_\gamma$ be coherent,
where $\gamma\leq\lambda$.  Then
$M_{\overline{\boldsymbol s}}$ is $\theta$-free as an $R$-module.
More precisely, every subset of cardinality $<\theta$ is contained
in a free $R$-submodule supported by fewer than $\theta$ branches.
The construction can be made relative: if a free submodule $P_0$
has already been obtained at some stage of this support construction
and $X\subseteq M_{\overline{\boldsymbol s}}$ has size $<\theta$,
then there is a free $R$-submodule $Q$ such that
\[
P_0\cup X\subseteq Q\subseteq M_{\overline{\boldsymbol s}}
\quad\text{and}\quad
Q=P_0\oplus Q'
\]
for some free $R$-module $Q'$.  In particular, if
$z\in Z^\sharp$ and the finite predecessor support of $z$ has been
included, the construction may be started with the direct summand
$Rz$ of $M_*$.
\end{lemma}

\begin{proof}
Null coordinates may be discarded.  Let $X$ have cardinality
$<\theta$.  Take all branch supports needed for $X$ and close them
under the finite predecessor supports occurring in the genuine
coordinates of $\overline{\boldsymbol s}$.  We obtain
$u\in[S]^{<\theta}$.  Enumerate the relevant genuine coordinates as
$\langle\beta_\xi:\xi<\tau\rangle$, where $\tau<\theta$.
By the $\theta$-freeness of $\bar\eta$, choose $i_\xi<\kappa$ so
that the tails
\(
\{\eta_{\beta_\xi}\upharpoonright j:j\geq i_\xi\}\) with
\(\xi<\tau,
\)
are pairwise disjoint and avoid the finitely many predecessor
supports already fixed at the $\xi$-th stage.  If
$\boldsymbol s_{\beta_\xi}=\langle\beta_\xi,\iota_\xi,
\zeta_\xi,z_\xi\rangle$, put
\(
L_\xi:=h^*_{\eta_{\beta_\xi},\zeta_\xi,z_\xi,\iota_\xi}
       (G^*_{i_\xi}).
\)
The free-amalgamation clauses of Construction~\ref{c310} and coherence
give, at the stage at which $N_{\boldsymbol s_{\beta_\xi}}$ is
added,
\(
N_{\boldsymbol s_{\beta_\xi}}\cap P_\xi=L_\xi.
\)
Consequently
\(
\frac{P_\xi+N_{\boldsymbol s_{\beta_\xi}}}{P_\xi}
 \cong
\frac{H^*_{\zeta_\xi}}{G^*_{i_\xi}},
\)
which is free by Definition~\ref{d65}(e).  Hence each successor
extension splits.  Choosing bases increasingly and taking unions at
limits yields a free module containing $X$.

The same induction proves the relative assertion: begin with a
basis of $P_0$ and, at every successor stage, extend that basis
across the split quotient.  Thus $P_0$ remains a direct summand of
the final free module $Q$.  Finally, Fact~\ref{unimodular} shows
that for $z\in Z^\sharp$ we may first replace a basis element by $z$
and start the construction with $Rz$.
\end{proof}

\begin{definition}\label{mulm}
Let $M$ be an $R$-module and let $h\in\End(M,+)$ be an additive
endomorphism. Let $r\in R$. By $\mu_r:M\to M$ we mean $\mu_r(m)=rm$. We say that $h$ is of \emph{multiplication type} if
there is $a\in R$ such that $h=\mu_a$; otherwise
$h$ is called a bad endomorphism.
\end{definition}

	We are now in a position to show that the construction can be continued while blocking unwanted endomorphisms.
\begin{lemma}\label{d89a}
Let $\overline{\boldsymbol s}\in\mathcal A_\gamma$ be coherent,
where $\gamma\leq\lambda$, and let
\(
h\in\End(M_{\overline{\boldsymbol s}},+)
\)
be not of multiplication type.  Then there is
\(
z=rx+y\in Z^\sharp\) with
\(r\in R,\ x\neq y\in\mathbf I
\)
such that $h(z)\notin Rz$.
\end{lemma}

\begin{proof}
Assume toward a contradiction that the membership
\(
h(rx+y)\in R(rx+y)
\) holds
for every $r\in R$ and every pair of distinct
$x,y\in\mathbf I$.  We first prove that $h$ is multiplication by one
fixed scalar on $M_*$.  For $y\in\mathbf I$, choose
$x\in\mathbf I\setminus\{y\}$ and put $r=0$ in the membership \(
h(rx+y)\in R(rx+y)
\).  Thus
$h(y)\in Ry$, so there is a unique $t_y\in R$ with
$h(y)=t_yy$.  Applying again the membership  property to $x+y$ and comparing the two
basis coordinates shows $t_x=t_y$.  Hence there is a single
$t\in R$ such that
\[
h(x)=tx\qquad(x\in\mathbf I)\quad(\ast)
\]

Fix distinct $x,y\in\mathbf I$.  For $r\in R$, write, using
the membership property,
\(
h(rx+y)=a_r(rx+y).
\)
By additivity and $(\ast)$,
\[
h(rx)=a_rrx+(a_r-t)y\quad(\ast,\ast)
\]
Put $b_r=a_r-t$.  Comparing $(\ast,\ast)$ for $r$, $s$ and $r+s$
gives
\[
b_{r+s}=b_r+b_s
\quad\text{and}\quad
b_rs+b_sr=0 \quad(+)
\]
For $r=1$, $(\ast)$ applied to $x+y$ gives $a_1=t$, hence
$b_1=0$.  Taking $s=1$ in $(+)$ yields $b_r=0$ for every
$r\in R$.  Therefore
\(
h(rx)=trx\) with  \(r\in R,\ x\in\mathbf I.
\)
Since $\mathbf I$ is a basis of $M_*$,
\[
h\upharpoonright M_*=\mu_t\upharpoonright M_*\quad (\sharp)
\]

It remains to justify the step that was missing in the previous
proof: $(\sharp)$ forces equality on every attached block as well.
Set $k=h-\mu_t$.  Then $k\upharpoonright M_*=0$.  Let
$\beta\in S\cap\gamma$ be a genuine coordinate and write
\[
N_\beta=M^*_{\eta_\beta,\zeta_\beta,z_\beta,\iota_\beta},
\qquad
e_\beta=h^*_{\eta_\beta,\zeta_\beta,z_\beta,\iota_\beta}:
H^*_{\zeta_\beta}\longrightarrow N_\beta.
\]
Because $e_\beta(G^*_\kappa)\subseteq M_*$, the additive map
$k\circ e_\beta$ vanishes on $G^*_\kappa$ and therefore factors as
\(
\bar k_\beta:
H^*_{\zeta_\beta}/G^*_\kappa
\longrightarrow M_{\overline{\boldsymbol s}}.
\)
Suppose $\bar k_\beta\neq0$.  Its image has cardinality at most
$|H^*_{\zeta_\beta}|\leq\chi<\theta$.  By Lemma~\ref{d222}, that
image is contained in a free $R$-submodule $P$ of
$M_{\overline{\boldsymbol s}}$.  Choose
$0\neq p\in\operatorname{Im}(\bar k_\beta)$.  Since $P$ is free,
a coordinate projection $\pi:P\to R$ may be chosen with
$\pi(p)\neq0$.  Then
\[
\pi\circ\bar k_\beta:
(H^*_{\zeta_\beta}/G^*_\kappa,+)\longrightarrow(R,+)
\]
is a nonzero additive homomorphism, contradicting clause (f) in
Definition~\ref{d65}.  Thus $\bar k_\beta=0$ for every
genuine coordinate $\beta$.
The module $M_{\overline{\boldsymbol s}}$ is generated by $M_*$
and these attached blocks.  Hence $k=0$ on all of
$M_{\overline{\boldsymbol s}}$, so $h=\mu_t$, contrary to the
assumption that $h$ is not of multiplication type.
\end{proof}

\begin{proposition}\label{12}
Assume that
\begin{enumerate}[(a)]
\item $\alpha\in S$;
\item $\overline{\boldsymbol s}\in\mathcal A_\alpha$ is coherent;
\item $h\in\End(M_{\overline{\boldsymbol s}},+)$ is not of
multiplication type;
\item
\(
z=r_zx_z+y_z\in Z^\sharp,
 x_z\neq y_z\in\mathbf I,
\) and  \(h(z)\notin Rz;
\)
\item let $d_z=(r_z,1)$ and choose
$\zeta<\zeta_{\boldsymbol f}$ with $d_\zeta=d_z$;
\item let
\(
\sigma_z:Rd_\zeta\to Rz\)
be defined by \(a d_\zeta\mapsto az,
\)
and set $g_{\zeta,z}:=\sigma_z\circ g_\zeta$;
\item for $\iota=0,1$ let
$\boldsymbol s_\iota=\langle\alpha,\iota,\zeta,z\rangle$ and
\(
M_\iota=M_{\overline{\boldsymbol s}^\frown
                 \langle\boldsymbol s_\iota\rangle};
\)
\item both one-step extensions are coherent and their embeddings
\(
e_\iota:=h^*_{\eta_\alpha,\zeta,z,\iota}:
H^*_\zeta\to M^*_{\eta_\alpha,\zeta,z,\iota}
\)
satisfy
\(
e_\iota(x)=h^*_{\eta_\alpha}(x)+\iota g_{\zeta,z}(x)\) for any
\(x\in G^*_\kappa .
\)
\end{enumerate}
Then for some $\iota<2$ the following stronger, persistent
non-extension statement holds: there are no
$\gamma$ with $\alpha<\gamma\leq\lambda$, no coherent
$\overline{\boldsymbol t}\in\mathcal A_\gamma$ extending
$\overline{\boldsymbol s}^\frown\langle\boldsymbol s_\iota\rangle$,
and no
\(
\widetilde h\in\End(M_{\overline{\boldsymbol t}},+)
\)
with
$\widetilde h\upharpoonright M_{\overline{\boldsymbol s}}=h$:

 \[
\xymatrix{
	&& M_{\overline{\boldsymbol{s}}} \ar[d]_{h} \ar[r]^{\subseteq} & M_{\overline{\boldsymbol{s}}^\frown \langle \boldsymbol{s}_\iota \rangle} \ar[d]^{\nexists\widetilde h} \\
	&& M_{\overline{\boldsymbol{s}}} \ar[r]^{\subseteq} & M_{\overline{\boldsymbol{s}}^\frown \langle \boldsymbol{s}_\iota \rangle},
}
\]
In particular, $h$ does not extend to an endomorphism of $M_\iota$.
\end{proposition}

\begin{proof}
First observe that $\sigma_z$ really is an $R$-isomorphism.  If
$a(r_z,1)=0$ in $R^2$, then its second coordinate gives $a=0$;
and if $az=0$, comparison of the coefficient of the basis element
$y_z$ likewise gives $a=0$.  Thus both $Rd_z$ and $Rz$ are free
cyclic modules with the displayed generators.
Assume, toward a contradiction, that neither successor has the
persistent non-extension property.  Then, for $\iota=0,1$, there
are $\gamma_\iota>\alpha$, coherent
$\overline{\boldsymbol t}_\iota\in\mathcal A_{\gamma_\iota}$
extending
$\overline{\boldsymbol s}^\frown\langle\boldsymbol s_\iota\rangle$,
and
\(
h_\iota\in\End(M_{\overline{\boldsymbol t}_\iota},+)
\)
with
$h_\iota\upharpoonright M_{\overline{\boldsymbol s}}=h$.
Put
$$
k_\iota=h_\iota\circ e_\iota:
(H^*_\zeta,+)\longrightarrow
(M_{\overline{\boldsymbol t}_\iota},+).
$$
Since $|G^*_\kappa|,|R|\leq\chi<\theta$, the set
\[
X_0=\{z,h(z)\}\cup h[M_{\eta_\alpha}]\cup h[Rz]
\]
has size $<\theta$.  By Fact~\ref{unimodular}, $Rz$ is a direct
summand of $M_*$.  Apply the relative form of Lemma~\ref{d222},
starting with $Rz$, to obtain a free submodule
$P\subseteq M_{\overline{\boldsymbol s}}$ such that
$X_0\subseteq P$ and
\(
P=Rz\oplus P'\)
for a free $R$-module $P'$.
Write
\[
h(z)=az+p_z,
\qquad a\in R,\quad 0\neq p_z\in P'.
\]
Choose an $R$-homomorphism $\pi:P'\to R$ with
$\pi(p_z)\neq0$, and define
\[
\delta:P\longrightarrow Rz,
\qquad
\delta(bz+p)=\pi(p)z.
\]
Then $\delta\upharpoonright Rz=0$ and
$\delta(h(z))\neq0$.
For $\iota=0,1$, the set $k_\iota[H^*_\zeta]$ has size at most
$\chi$.  Apply again the relative form of Lemma~\ref{d222} inside
$M_{\overline{\boldsymbol t}_\iota}$, now starting with $P$, to
obtain a free
$P_\iota\subseteq M_{\overline{\boldsymbol t}_\iota}$ with
\[
P\cup k_\iota[H^*_\zeta]\subseteq P_\iota,
\qquad
P_\iota=P\oplus Q_\iota.
\]
Extend $\delta$ by zero on $Q_\iota$ and denote the resulting map
by $\delta_\iota:P_\iota\to Rz$.
Define the additive homomorphism
\[
\Phi=\sigma_z^{-1}\circ
      (\delta_1\circ k_1-\delta_0\circ k_0):
H^*_\zeta\longrightarrow Rd_\zeta\subseteq R^2.
\]
Let $x\in G^*_\kappa$.  By item $(h)$  from our hypothesis,
\[
e_0(x)=h^*_{\eta_\alpha}(x),
\qquad
e_1(x)=h^*_{\eta_\alpha}(x)+g_{\zeta,z}(x).
\]
Both elements lie in the old module
$M_{\overline{\boldsymbol s}}$: the first is in
$M_{\eta_\alpha}\subseteq M_*$ and the second differs from it by
an element of $Rz\subseteq M_*$.  Therefore
\(
k_1(x)-k_0(x)=h(g_{\zeta,z}(x)).
\)
Moreover, by the choice of $P$, both $k_0(x)$ and $k_1(x)$ lie in
$P$, so $\delta_0$ and $\delta_1$ agree with $\delta$ there.  Hence
\(
\Phi(x)=
\sigma_z^{-1}\delta\bigl(h(g_{\zeta,z}(x))\bigr).
\)
Since $g_{\zeta,z}=\sigma_z\circ g_\zeta$, define $\tau$ by the following composition:
$$Rd_\zeta\stackrel{ \sigma_z}\longrightarrow Rz \stackrel{h\rest}\longrightarrow P\stackrel{\delta}\longrightarrow Rz\stackrel{\sigma_z^{-1}}\longrightarrow Rd_\zeta\subseteq (R^2,+),$$
i.e.,
\(
\tau=\sigma_z^{-1}\circ\delta\circ h\circ\sigma_z.
\)
This is well-defined because $h[Rz]\subseteq P$.  We have
\(
\Phi\upharpoonright G^*_\kappa=\tau\circ g_\zeta:
\) \[
\begin{CD}
G^*_\kappa @>\subseteq>> H^*_\zeta  \\
@Vg_\zeta VV @VV\Phi V \\
Rd_\zeta @>\tau>> R\times R
\end{CD}
\]
Finally, since $\sigma_z(d_\zeta)=z$,
\[
\tau(d_\zeta)=\sigma_z^{-1}(\delta(h(z)))\neq0.
\]
Thus $\Phi$ is an additive extension to $H^*_\zeta$ of
$\tau\circ g_\zeta$ with $\tau(d_\zeta)\neq0$, contradicting
clause (i) in the definition of a nice frame.  Therefore at least
one of the two successors has the persistent non-extension property: no endomorphism
of any later coherent extension through that successor can extend
$h$.  This proves the proposition.
\end{proof}

We now state and prove the main theorem of this section.

\begin{theorem}\label{d71}
Assume Hypothesis~\ref{hp71}.  Then there exists a $\theta$-free
$R_{\boldsymbol f}$-module $M$ such that
\(
\End_{\mathbb Z}(M)\cong R_{\boldsymbol f}.
\)
If, in addition, $(R_{\boldsymbol f},+)$ is $\theta$-free, then
$(M,+)$ is a $\theta$-free abelian group.
\end{theorem}

\begin{proof}
Put $R=R_{\boldsymbol f}$.  We first record the smallness needed
for weak diamond.  By $(e_2)$ this is immediate, while under
$(e_1)$ every member of $\mathcal T$ is a sequence of length
$<\kappa$ from $U$, so
$
|\mathcal T|\leq\mu^{<\kappa}<\lambda.
$
Since every $G_i^*$ has size at most $\chi<\lambda$, Construction
\ref{c310} gives
\(
|M_*|<\lambda.
\)
Each attached module has size at most $\chi$, and therefore every
module obtained from $M_*$ by fewer than $\lambda$ attachments has
cardinality $<\lambda$.

Here, we connect to stationary sets and the witnesses. To this end,
for each $z\in Z^\sharp$ fix once and for all a presentation
\[
z=r_zx_z+y_z,
\qquad r_z\in R,\quad x_z\neq y_z\in\mathbf I\quad(\sharp)
\]
Put $d_z=(r_z,1)$ and choose $\zeta(z)<\zeta_{\boldsymbol f}$
with $d_{\zeta(z)}=d_z$.  By Construction~\ref{c310}(c),
$(\zeta(z),z)$ is compatible and we have
\[
g_z:=g_{\zeta(z),z}
 =\sigma_z\circ g_{\zeta(z)}:G^*_\kappa\longrightarrow Rz,
\quad
\sigma_z(a d_z)=az.
\]

Recall from \(
|M_*|<\lambda
\) that $|\mathbf I|<\lambda$; hence
$|Z^\sharp|<\lambda$.  Choose an injection
$\nu:Z^\sharp\to\lambda$ and set
\(
S_z=S_{\nu(z)}.
\)
The $S_z$ are pairwise disjoint stationary sets and weak diamond
holds on each of them.
For each $z\in Z^\sharp$, choose a finite branch support
$u(z)\subseteq S$ for the two basis elements occurring in $(\sharp)$
and put
\(
b(z)=\sup u(z)+1.
\)
Thus whenever $\alpha>b(z)$, attaching the test belonging to $z$
at coordinate $\alpha$ satisfies the coherence requirement of
Definition~\ref{83}.

We are going to apply the binary tree of approximations.
Indeed, we recursively construct, for every $\rho\in{}^{<\lambda}2$, a
coherent
\(
\overline{\boldsymbol s}_\rho\in\mathcal A_{\lg(\rho)}
\)
so that the system is increasing under initial segment.  At limit
lengths we take unions.  Suppose $\lg(\rho)=\alpha$.
If $\alpha\in S_z$ and $\alpha>b(z)$ for the unique possible
$z\in Z^\sharp$, define \(
\boldsymbol t_{\alpha,z,\iota}
:=\langle\alpha,\iota,\zeta(z),z\rangle
\)  for $\iota<2$,
and put
\(
\overline{\boldsymbol s}_{\rho^\frown\langle\iota\rangle}
 =\overline{\boldsymbol s}_\rho{}^\frown
  \langle\boldsymbol t_{\alpha,z,\iota}\rangle.
\)
If $\alpha\in S$ but this is not an active stage, append the null
symbol $\mathbf0_\alpha$ to both successors.  If $\alpha\notin S$,
there is no new $S$-coordinate and we keep the same sequence for
both successors.  This defines the whole tree.  Set
\(
M_\rho=M_{\overline{\boldsymbol s}_\rho}.
\)
By \(
|M_*|<\lambda
\), regularity of $\lambda$, and the size bound on the
attachments,
\[
|M_\rho|<\lambda
\qquad(\rho\in{}^{<\lambda}2).\]
At an active stage, Proposition~\ref{12} applies to every additive
endomorphism $k$ of $M_\rho$ satisfying $k(z)\notin Rz$: one of the
two successors in $\overline{\boldsymbol s}_{\rho^\frown\langle\iota\rangle}
=\overline{\boldsymbol s}_\rho{}^\frown
\langle\boldsymbol t_{\alpha,z,\iota}\rangle$ does not admit an extension of $k$.
There is a fixed coding of the tree such that:
\begin{quote}
for every branch $g\in{}^\lambda2$ and every
$h\in\End(\bigcup_{\alpha<\lambda}M_{g\upharpoonright\alpha},+)$,
there is $F\in{}^\lambda2$ and a club $C\subseteq\lambda$ such that,
whenever $\alpha\in C$ and
$h[M_{g\upharpoonright\alpha}]\subseteq
 M_{g\upharpoonright\alpha}$, the initial segment
$F\upharpoonright\alpha$ decodes exactly the pair
\[
\bigl(g\upharpoonright\alpha,
 h\upharpoonright M_{g\upharpoonright\alpha}\bigr)\quad
(\dagger)
\]
\end{quote}
Indeed, first enumerate the fixed core $M_*$ by ordinals below
some $\gamma_*<\lambda$.  Every later element has a canonical code
consisting of the stage at which its attachment was made together
with an index $<\chi$ in a fixed copy of $H^*_\zeta$.  Fix pairing
and finite-sequence functions on $\lambda$.  The ordinals closed
under these functions and above $\gamma_*$ form a club.  On those
ordinals the elements of $M_\rho$, the branch $\rho$, and the graph
of a map $M_\rho\to M_\rho$ are coded below the same ordinal.
Interleaving the code of the branch with the code of the graph
gives $(\dagger)$.  This is exactly where the strict bound
$|\mathcal T|<\lambda$ is needed.

We shall also use the standard closure fact that for a continuous
increasing chain $\langle M_\alpha:\alpha<\lambda\rangle$ with
$|M_\alpha|<\lambda$, every self-map of the union leaves
$M_\alpha$ invariant for club many $\alpha$.  To see this, for each
$\alpha$ choose $f(\alpha)<\lambda$ with
$h[M_\alpha]\subseteq M_{f(\alpha)}$; the closure points of $f$
form a club, and continuity gives the assertion.

\medskip
Let us apply the weak-diamond colorings. First,
for any fixed $z\in Z^\sharp$, we define
$c_z:{}^{<\lambda}2\to2$ as follows.  If $w\in{}^\alpha2$ with
$\alpha\in S_z$, $\alpha>b(z)$, and $w$ decodes a pair
$(\rho,k)$ where
\[
\rho\in{}^\alpha2,
\qquad
k\in\End(M_\rho,+),
\qquad
k(z)\notin Rz,
\]
then put  $c_z(w)$ to be
\[
\begin{aligned}
\min\{\iota<2:
\text{the $\iota$-th successor has
  non-extension property of Proposition~\ref{12} for }k\}.
\end{aligned}
\]

The set in $c_z(w)$ is nonempty by Proposition~\ref{12}; note
that $k(z)\notin Rz$ already implies that $k$ is not of
multiplication type.  On all invalid or inactive codes put
$c_z(w)=0$.
By $\Phi_\lambda(S_z)$ choose
\(
\varepsilon_z:S_z\to2
\)
that predicts $c_z$ in the sense of weak diamond.  Define one branch
$g\in{}^\lambda2$ by
\[
g(\alpha)=
\begin{cases}
\varepsilon_z(\alpha),&\alpha\in S_z
 \text{ for some }z\in Z^\sharp,\\
0,&\text{otherwise}.
\end{cases}
\quad(\ast)
\]
This is well-defined because the $S_z$ are pairwise disjoint.
Finally put
\[
\overline{\boldsymbol s}
 =\bigcup_{\alpha<\lambda}
   \overline{\boldsymbol s}_{g\upharpoonright\alpha},
\qquad
M=M_{\overline{\boldsymbol s}}
 =\bigcup_{\alpha<\lambda}M_{g\upharpoonright\alpha}.
\]
The sequence $\overline{\boldsymbol s}$ is coherent, and therefore
Fact~\ref{d222} implies that $M$ is $\theta$-free as an
$R$-module.

\medskip
We now proceed with the proof of our theorem by eliminating bad endomorphisms.
Suppose toward a contradiction that $h\in\End(M,+)$ is not of
multiplication type.  By Lemma~\ref{d89a}, choose
$z\in Z^\sharp$ with
\(
h(z)\notin Rz.
\)
Apply the coding claim to the branch $g$ and the map $h$, and
intersect its club with the club of stages at which
\(
h[M_{g\upharpoonright\alpha}]
 \subseteq M_{g\upharpoonright\alpha}.
\)
Let $F\in{}^\lambda2$ be the resulting global code.  Weak diamond
on $S_z$ (see Definition~\ref{dwd}(2)) gives stationarily many $\alpha\in S_z$ such that
\[
c_z(F\upharpoonright\alpha)=\varepsilon_z(\alpha)
\quad(+)
\]
Choose such an $\alpha$ in the two clubs and above $b(z)$.  Put
\(
\rho=g\upharpoonright\alpha,\) and
\(
k=h\upharpoonright M_\rho.
\)
Since $h[M_{g\upharpoonright\alpha}]
\subseteq M_{g\upharpoonright\alpha}$ we observe that $k\in\End(M_\rho,+)$. Since $h(z)\notin Rz$, we deduce that
$k(z)\notin Rz$; and by $(\dagger)$,
$F\upharpoonright\alpha$ decodes precisely $(\rho,k)$.  Hence, combining the definition of $c_z(-)$ with
$(\ast)$ and $(+)$, we deduce that $g(\alpha)$ is equal to
\[
\min\{\iota<2:
 \text{the $\iota$-th successor has
 non-extension property of Proposition~\ref{12} for }k\}.
\]
But the final coherent sequence $\overline{\boldsymbol s}$ extends
$\overline{\boldsymbol s}_\rho{}^\frown
 \langle\boldsymbol t_{\alpha,z,g(\alpha)}\rangle$, and the global
endomorphism $h$ extends $k$ on $M_\rho$.  This is exactly what the last
non-extension property forbids.  Hence we
obtain a contradiction.  Thus every additive endomorphism of $M$
is of multiplication type.
For $r\in R$, multiplication
$\mu_r(x)=rx$ is an additive endomorphism of $M$, and
\(
\varphi:R\to\End(M,+),\) with
\( r\mapsto \mu_r,
\)
is a ring homomorphism.  The preceding paragraph proves
surjectivity.  It is injective because $M$ contains the free core
$M_*$: choosing a basis element $x\in\mathbf I$, the equality
$\mu_r=0$ gives $rx=0$, hence $r=0$.  Therefore
\(
R\cong\End_{\mathbb Z}(M).
\)
Finally assume that $(R,+)$ is $\theta$-free and let
$A\subseteq(M,+)$ have cardinality $<\theta$.  By the $R$-module
$\theta$-freeness just proved, $A$ is contained in a free
$R$-submodule
\(
P=\bigoplus_{b\in B}Rb.
\)
Only fewer than $\theta$ coordinates of $B$ occur in elements of
$A$; call this set $B_0$.  For each $b\in B_0$, let $A_b$ be the
subgroup of $(R,+)$ generated by the $b$-coordinates of elements
of $A$.  Then $|A_b|<\theta$, so $A_b$ is free abelian.  Thus
$\bigoplus_{b\in B_0}A_b$ is free abelian and contains $A$.
Every subgroup of a free abelian group is free, so $A$ is free.
Hence $(M,+)$ is $\theta$-free.
\end{proof}

\section{Realization via the Super Black Box}
\label{sec:sbb-realization}

The main result of this section is {Theorem} \ref{main2}.

\begin{notation} Let $\kappa<\lambda$ be regular cardinals and $S \subseteq \lambda$ be stationary. We set $S^\lambda_{\kappa}=\{\alpha < \lambda\mid \cf(\alpha)=\kappa     \}$.
\end{notation}
\begin{definition}\label{sbb-free}
	Let $\kappa<\lambda$ be regular cardinals and let
	$S\subseteq S^\lambda_\kappa$.  A family
	\[
	\bar C=\langle C^\delta_\gamma:\delta\in S,\ \gamma<\lambda\rangle
	\]
	is called \emph{$\mu^+$-free} if whenever
	$u\subseteq S\times\lambda$ has cardinality at most $\mu$, there are
	bounded sets $v^\delta_\gamma\subseteq C^\delta_\gamma$,
	$(\delta,\gamma)\in u$, such that
	\(
	\bigl\langle C^\delta_\gamma\setminus v^\delta_\gamma:
	(\delta,\gamma)\in u\bigr\rangle
	\)
	is pairwise disjoint.
\end{definition}

\begin{definition}\label{sbbd}
	Let $\mu$ be singular, $\kappa=\cf(\mu)$, let $\lambda>\mu$ be regular,
	and let $\Theta<\mu$.  We say that
	$\bar C=\langle C^\delta_\gamma:\delta\in S,\gamma<\lambda\rangle$
	has the \emph{multiple Super Black Box property with colors $\Theta$} if
	$S\subseteq S^\lambda_\kappa$ is stationary, each
	$C^\delta_\gamma\subseteq\delta$ is unbounded of order type $\kappa$,
	$\bar C$ is $\mu^+$-free, and for every family of colorings
	\[
	F^\delta_\gamma:{}^{C^\delta_\gamma}(2^\mu)\longrightarrow\Theta
	\qquad(\delta\in S,\ \gamma<\lambda)
	\]
	there are colors $c^\delta_\gamma<\Theta$ such that for every
	$\delta\in S$ and every $q:\delta\to2^\mu$, the equality
	\(
	F^\delta_\gamma(q\restriction C^\delta_\gamma)=c^\delta_\gamma
\)
	holds for at least one (in fact, in the version used here, for many)
	$\gamma<\lambda$.
\end{definition}

\begin{fact}[Shelah, \cite{Sh:1268}]\label{super-bbthm}
	Assume that $\mu$ is a strong limit singular cardinal,
	$\kappa=\cf(\mu)$, $\Theta<\mu$, and
	\(
	\lambda=\min\{\rho:2^\rho>2^\mu\}<2^\mu.
	\)
	Then $\lambda$ is regular.  For every stationary
	$S\subseteq S^\lambda_\kappa$ there is a family with the multiple Super
	Black Box property of Definition~\ref{sbbd}.  The construction may be
	applied independently to at most $\lambda$ many stationary sets.
\end{fact}

\begin{proof}
	The regularity of $\lambda$ follows from minimality.  If
	$\nu=\cf(\lambda)<\lambda$ and
	$\lambda=\bigcup_{i<\nu}I_i$ with $|I_i|=\rho_i<\lambda$, then
	$2^{\rho_i}\leq2^\mu$ for every $i<\nu$, and therefore
	\[
	2^\lambda\leq\prod_{i<\nu}2^{\rho_i}
	\leq(2^\mu)^\nu=2^{\mu\cdot\nu}\leq2^\mu,
	\]
	because $\max\{\mu,\nu\}<\lambda$ and, by the minimality of $\lambda$,
	$2^{\max\{\mu,\nu\}}\leq 2^\mu$.  This contradicts
	$2^\lambda>2^\mu$.
	The black-box assertion is the required specialization of the multiple
	Super Black Box of \cite{Sh:1268}; compare Theorem~0.8 and Claim~0.9(2)
	there, together with the freeness conclusion of Theorem~0.6.  Applying the
	result separately gives the last assertion.
\end{proof}

\begin{theorem}\label{main2}
	Assume that $\mu$ is a strong limit singular cardinal of cofinality
	$\aleph_0$ and
	\(
	\lambda=\min\{\rho:2^\rho>2^\mu\}<2^\mu.
	\)
	Let $R$ be a ring with $1$, $|R|<\mu$, whose additive group is
	cotorsion-free and $\sigma$-free, where $\sigma$ is an infinite cardinal or
	$0$.  Then there exists a $\mu^+$-free $R$-module $M$ of cardinality
	$\lambda$ such that
	\(
	\End_{\mathbb Z}(M)\cong R.
	\)
	Moreover $(M,+)$ is $\min\{\sigma,\mu^+\}$-free, with the usual convention
	that this assertion is void when $\sigma=0$.
\end{theorem}

\begin{proof}
	Put
	\(
	\Theta=\max\{|R|,\aleph_0\}<\mu.
	\)
	We divide the proof into steps.
	
	\medskip
	\noindent\textit{Step 1: the free core and stationary bookkeeping.}
	
	Choose a free $R$-module
	\[
	X=\bigoplus_{x\in\mathbf I}Rx,
	\qquad |\mathbf I|=\lambda.
	\]
	Put
	\(
	Z=\{rx+y:r\in R,\ x,y\in\mathbf I,\ x\neq y\}.
	\)
	Then $|Z|=\lambda$, and every $b=rx+y\in Z$ is unimodular in $X$: replacing
	$y$ by $b$ gives another basis of $X$.  Hence $Rb$ is a free cyclic direct
	summand of $X$. We shall repeatedly use the following elementary fact.
	
	\begin{fact}  If
	$N\supseteq X$ is an $R$-module and $g:X\to N$ is additive such that
	$g(b)\in Rb$ for every $b\in Z$, then $g$ is multiplication by a single scalar
	from $R$ on $X$.  Indeed, applying the hypothesis to basis elements and to
	$x+y$ shows that $g(x)=rx$ for one fixed $r\in R$ and every basis element
	$x$.  If $r_0\in R$ and $x,y,z$ are three distinct basis elements, then
	\[
	g(r_0x+y)=s(r_0x+y),\qquad
	g(r_0x+z)=t(r_0x+z)
	\]
	for suitable $s,t\in R$.  Subtracting $g(y)=ry$ and $g(z)=rz$ and comparing
	coefficients in the free module $X$ gives $s=t=r$, hence
	$g(r_0x)=rr_0x$.  Additivity finishes the argument. \end{fact}
	By the stationary splitting theorem, partition a stationary
	$S\subseteq S^\lambda_\omega$ into pairwise disjoint stationary sets
	\(
	S=\bigcup_{b\in Z}S_b.
	\)
	For each $b\in Z$ apply Fact~\ref{super-bbthm} to $S_b$ and obtain
	\(
	\bar C^b=\langle C^{b,\delta}_\gamma:
	\delta\in S_b,\ \gamma<\lambda\rangle.
	\)
	For notational simplicity, when $\delta\in S_b$ we write
	$C^\delta_\gamma=C^{b,\delta}_\gamma$ and put $b_\delta=b$.
	After intersecting each $S_b$ with a tail, we may assume that all basis
	coordinates occurring in $b$ have rank below every $\delta\in S_b$.
	
	To keep the branch supports belonging to different witnesses disjoint,
	partition $\mathbf I$ into sets
	\(
	\mathbf I_b\quad(b\in Z)\),		\( |\mathbf I_b|=\lambda,
	\)
	and, deleting the finite support of $b$ from $\mathbf I_b$, fix bijections
	\(
	\iota_b:\lambda\longrightarrow
	\mathbf I_b\setminus\supp_{\mathbf I}(b).
	\)
	The rank of $\iota_b(\alpha)$ is declared to be $\alpha$.  We also assign
	ranks below $\lambda$ to the finitely many remaining basis elements so that
	the sets of rank $<\xi$ have cardinality $<\lambda$ for $\xi<\lambda$.
	Let $X_{<\xi}$ be the free submodule generated by the basis elements of rank
	$<\xi$.
	
	\medskip
	\noindent\textit{Step 2: the cotorsion-free test systems.}

	Let $0\neq c\in R$.  Regard $R$ as a subgroup of its $\mathbb Z$-adic
	completion $\widehat R$.  Since $(R,+)$ is cotorsion-free, there is
	$\pi_c\in\widehat{\mathbb Z}$ such that
	\(
	\pi_c c\notin R.
	\)
	Indeed, otherwise the nonzero subgroup $\widehat{\mathbb Z}c$ of $R$ would
	be cotorsion.  Put $k_n=n+2$ and
	$K_0=1$, $K_{n+1}=k_nK_n=(n+2)!$.  Choose the factorial expansion
	\(
	\pi_c=\sum_{n<\omega}\ell_n^cK_n,\)
	\(0\leq\ell_n^c<k_n.
	\)
	Then the system
	\[
	u_n=k_nu_{n+1}+\ell_n^c c
	\qquad n<\omega
\quad(\ast)
	\]
	has no solution in $R$.  For if it had a solution, iteration would give
	\(
	u_0=K_m u_m+\sum_{n<m}\ell_n^cK_nc.
	\)
	As $m\to\infty$, the first term tends to $0$ in $\widehat R$, so
	$u_0=\pi_cc$, contrary to \(
	\pi_c c\notin R
	\).
	Fix $\delta\in S_b$ and $\gamma<\lambda$, and enumerate
	\(
	C^\delta_\gamma=\langle\alpha_n:n<\omega\rangle
	\)
	in increasing order.  Put
	\(
	x^{b}_{\alpha_n}=\iota_b(\alpha_n)\in\mathbf I_b.
	\)
\begin{definition}\label{defy}
	For $0\neq c\in R$ and $i<2$, define elements
$y_n^{\delta,\gamma,c,i}\in\widehat X$ by the convergent tail series
\[
y_n^{\delta,\gamma,c,i}
=\sum_{m\geq n}\Bigl(\prod_{n\leq j<m}k_j\Bigr)
\bigl(x^b_{\alpha_m}+i\ell_m^c b\bigr).
\]\end{definition}
	Thus
	\[
	y_n^{\delta,\gamma,c,i}
	=k_ny_{n+1}^{\delta,\gamma,c,i}
	+x^b_{\alpha_n}+i\ell_n^c b
	\quad(\sharp)
	\]
	Let $\Gamma^{\delta,\gamma}_{c,i}$ be the $R$-submodule of $\widehat X$
	generated by
	$\{
	b,  x^b_{\alpha_n}, 	y_n^{\delta,\gamma,c,i}:  n<\omega
\}$.
	The displayed realization has no additional finite $R$-linear relations.
	Indeed, if
	$r b+\sum_{n\leq N}r_ny_n^{\delta,\gamma,c,i}=0$, comparison of the
	coordinate $x^b_{\alpha_0}$ gives $r_0=0$, then comparison of
	$x^b_{\alpha_1}$ gives $r_1=0$, and so on; finally $r=0$.  Consequently
	\(
	\Gamma^{\delta,\gamma}_{c,i}
	\cong Rb\oplus\bigoplus_{n<\omega}Ry_n^{\delta,\gamma,c,i}.
	\)
	Moreover,
	\[
	\Gamma^{\delta,\gamma}_{c,i}/
	\Bigl(Rb+\sum_{n<m}Rx^b_{\alpha_n}\Bigr)
	\quad\hbox{is free}
\quad(\ast,\ast)	\]
for every $m<\omega$. Indeed, modulo the displayed root the first $m$ relations in $(\sharp)$
	allow successive elimination of $\{y_0,\dots,y_{m-1}\}$.
	Set
	\[
	E=(R\setminus\{0\})\times2,
	\qquad
	Y^\delta_\gamma=
	\{\Gamma^{\delta,\gamma}_{c,i}:(c,i)\in E\}.
	\]
	Thus $|E|\leq\Theta<\mu$.
\begin{notation}\label{gf} We use the following notation.
\begin{enumerate}
	\item[$\bullet$]  For $\xi\leq\lambda$ let $\mathbf P_\xi$ consist of the functions
	\(
	f:\{ y_{\delta, \gamma, 0}: (\delta, \gamma) \in (S\cap\xi)\times\xi\}\to E,
	\) and
let $\mathbf P_\xi^+$ consist of the functions
	\(
	f:\{ y_{\delta, \gamma, 0}: (\delta, \gamma) \in (S\cap\xi)\times\lambda\}\to E.
	\)
	\item[$\bullet$] For $f\in\mathbf P_\xi$, let $G_f$ be the $R$-submodule of $\widehat X$
	generated by $X_{<\xi}$ together with all
	$\Gamma^{\delta,\gamma}_{f(y_{\delta,\gamma, 0})}$ for
	$(\delta,\gamma)\in\dom(f)$.
The same definition is used for
	$f\in\mathbf P_\xi^+$.
\end{enumerate}\end{notation}

	\begin{claim}\label{freeness}
Let $f\in\mathbf P_\xi$, and suppose $P_0$ is one of the supported free
members of the system, and let $A\subseteq G_f$ be of cardinality $\leq\mu$. Then there is a
supported free $P\subseteq G_f$ such that
\(P_0\cup A\subseteq P\),
		\(P=P_0\oplus P_1
		\)
		for a free $R$-module $P_1$.  Consequently $G_f$ is $\mu^+$-free and, if
		$\sigma\neq0$, $(G_f,+)$ is $\min\{\sigma,\mu^+\}$-free.
	\end{claim}
	
	\begin{PROOF}{\ref{freeness}}
		At most $\mu$ local pieces are needed to support $A$; call their set of
		coordinates $u$.  For each fixed $b\in Z$, let $u_b$ be the coordinates in
		$u$ whose first coordinate belongs to $S_b$.  The $\mu^+$-freeness of
		$\bar C^b$ gives, for $(\delta,\gamma)\in u_b$, finite initial pieces of
		$C^\delta_\gamma$ whose complementary tails are pairwise disjoint.  Since
		the tagged basis blocks $\mathbf I_b$ are pairwise disjoint, the resulting
		tails are pairwise disjoint for all members of $u$, including members
		belonging to different $b$'s.
		
		For $a=(\delta,\gamma)\in u$, choose $m_a<\omega$ beyond the corresponding
		finite deleted set and put
		\(
		B_a=Rb_\delta+\sum_{n<m_a}R x^{b_\delta}_{\alpha_n}.
		\)
		By $(\ast,\ast)$,
		$\Gamma_a/B_a$ is free, where $\Gamma_a=\Gamma^{\delta, \gamma}_{c, i}$ with $(c,i)=f(y_{\delta,\gamma, 0})$ denotes the local module selected
		at $a$.  Close the support of $A$ (and of $P_0$, in the relative case)
		under all these roots.  Add the selected local pieces one at a time.  At a
		successor step, the private $y$-coordinates and the disjoint tagged tails
		show that the intersection with what has already been constructed is exactly
		the chosen root together with the part already belonging to $P_0$; hence
		the new quotient is a free module.  Choose a basis extending the previous
		basis.  At limits take unions of the increasing bases.  This produces the
		required $P$, and in the relative construction $P_0$ remains a free direct
		summand.
		
		For the additive assertion put $\tau=\min\{\sigma,\mu^+\}$ and let
		$A_0\leq(G_f,+)$ have cardinality $<\tau$.  Put $A_0$ inside a free
		$R$-module $P=\bigoplus_{j\in J}Re_j$ supplied above.  Only $<\tau$ basis
		coordinates occur in elements of $A_0$.  For every such coordinate $j$,
		the subgroup of $(R,+)$ generated by its coefficients has cardinality
		$<\tau\leq\sigma$, hence is free abelian.  Therefore $A_0$ is a subgroup of
		a free abelian group and is itself free.  If $\tau=\aleph_0$, the conclusion
		follows from torsion-freeness.  This proves the claim.
	\end{PROOF}
	
	\medskip
	\noindent\textit{Step 3: non-extendability of bad homomorphisms.}
	
	\begin{claim}\label{pnon}
		Assume $\delta\in S_b$, $\gamma<\lambda$, $h\in\mathbf P_\delta$, and
		$g\in\End(G_h,+)$ is a bad homomorphism, i.e.,  it satisfies
		\(
		g(b)\notin Rb.
	\)
		Then there is $(c,i)\in E$ such that, whenever $f$ is any later selection
		extending $h$ and
		$f(y_{\delta,\gamma, 0})=(c,i)$, the map $g$ has no extension to an additive
		endomorphism of $G_f$:
		\[
		\begin{CD}
		G_h @>\subseteq>> G_f  \\
		@VgVV @VV\not\exists\overline{g}V \\
		G_h  @>\subseteq>> G_f
		\end{CD}
		\]
	\end{claim}
	
	\begin{PROOF}{\ref{pnon}}
		Write $C^\delta_\gamma=\langle\alpha_n:n<\omega\rangle$.  By the relative
		part of Claim~\ref{freeness}, starting with the free direct summand $Rb$,
		choose a supported free module
		\(
		P_0=Rb\oplus P_0'
		\)
		which contains
		\(
	\{	g(b),
		g(x^b_{\alpha_n}):n<\omega\}.
		\)
		Since $g(b)\notin Rb$, the $P_0'$-component of $g(b)$ is nonzero.  Choose a
		basis coordinate projection $d:P_0'\to R$ which is nonzero on that
		component, and extend it by zero on $Rb$.  Thus
		\(
		d(b)=0,	\) and
	\( c:=d(g(b))\neq0.
		\)
Next we call $i<2$ \emph{admissible} if there is some later selection $f_i$
		containing $\Gamma^{\delta,\gamma}_{c,i}$ and an endomorphism
		$\bar g_i\in\End(G_{f_i},+)$ extending $g$.  Suppose both $0$ and $1$ are
		admissible.  By the relative part of Claim~\ref{freeness}, enlarge $P_0$ in
		$G_{f_i}$ to a supported free direct sum
		\(
		P_i=P_0\oplus Q_i
		\)
		containing all $\bar g_i(y_n^{\delta,\gamma,c,i})$, $n<\omega$.
		Extend $d$ to an $R$-homomorphism $d_i:P_i\to R$ by setting it equal to zero
		on $Q_i$, and put
		\(
		a_n^i=d_i\bigl(\bar g_i(y_n^{\delta,\gamma,c,i})\bigr).
		\)
		Applying $\bar g_i$ to $(\sharp)$  and then $d_i$ gives
		\[
		a_n^i=k_na_{n+1}^i+d(g(x^b_{\alpha_n}))+i\ell_n^c d(g(b)).
		\]
		Subtracting the equation for $i=0$ from the equation for $i=1$ and putting
		$u_n=a_n^1-a_n^0$, we obtain
		\(
		u_n=k_nu_{n+1}+\ell_n^c c,
		\)
		contrary to $(\ast)$.  Hence at least one of the two bits is not admissible.
		For that $i$, the pair $(c,i)$ has the required persistent non-extension
		property.
	\end{PROOF}
	
	%Notice that scalar maps are never killed by this argument: if $g=\mu_r$, then
	%$d(g(b))=d(rb)=r d(b)=0$.
	
	\medskip
	\noindent\textit{Step 4: the Super Black Box choice at one row.}
	
	\begin{claim}\label{non}
		Assume $\xi\in S_b$ and $h\in\mathbf P_\xi$.  There is
		$h^+\in\mathbf P_{\xi+1}^+$ extending all previously fixed choices such that
		\[
		(*)_\xi\qquad
		\text{every }g\in\End(G_h,+)\text{ with }g(b)\notin Rb
		\text{ is permanently killed at some }(\xi,\gamma).
		\]
	\end{claim}
	
	\begin{PROOF}{\ref{non}}
		Since $|G_h|<\lambda$ and $\lambda$ is the first cardinal at which the
		power set exceeds $2^\mu$,
		\(
		|\End(G_h,+)|\leq2^{|G_h|}\leq2^\mu.
		\)
		Fix an injection
		\(
		\operatorname{code}:\End(G_h,+)\to 2^\mu.
		\)
		For every $\gamma<\lambda$ and every $g$ with $g(b)\notin Rb$, choose by
		Claim~\ref{pnon} a killing value
		\(
		e_g(\gamma)\in E.
		\)
		Fix an injection of $E$ into $\Theta$, and identify $E$ with its image.
		Define
		\(
		F^\xi_\gamma:{}^{C^\xi_\gamma}(2^\mu)\to\Theta
		\)
		by setting $F^\xi_\gamma(q)=e_g(\gamma)$ when $q$ is the constant function
		with value $\operatorname{code}(g)$ for some such $g$, and taking a fixed
		default value otherwise.  Applying the row consequence of
		Fact~\ref{super-bbthm} to the stationary system on $S_b$ gives colors
		$e^\xi_\gamma\in E$ such that, for every $g$ with $g(b)\notin Rb$, some
		$\gamma<\lambda$ satisfies
		\(
		e^\xi_\gamma=e_g(\gamma).
		\)
		Choose the row $\xi$ of $h^+$ by
		$h^+(y_{\xi,\gamma, 0})=e^\xi_\gamma$.  Claim~\ref{pnon} gives $(*)_\xi$.
	\end{PROOF}
	
	\medskip
	\noindent\textit{Step 5: the original recursion.}
	
		\begin{construction}\label{main-construction}
		We construct $\langle f_\xi, f_\xi^+ : \xi < \lambda \rangle$ by induction on $\xi$ satisfying the following conditions:
		
		\begin{enumerate}
			\item[(a)] $f_\xi \in \mathbf{P}_\xi$ and $f_\xi^+ \in \mathbf{P}_{\xi+1}^+$,
			
			\item[(b)] $\xi < \zeta \implies f_\xi \subseteq f_\zeta$ and $f_\xi^+ \subseteq f_\zeta^+$,
			
				\item[(c)] $\xi < \zeta \implies f_\xi \cup \big( f_\xi^+ \restriction \{y_{\delta,\gamma,0}: (\delta,\gamma) \in (S \cap \xi) \times \zeta\} \big) \subseteq f_\zeta$,
			
			\[
			\begin{CD}
			f_\xi^+(y_{\delta,\gamma,0}) @>\subseteq>> f_\zeta^+(y_{\delta,\gamma,0}) \\
			@AAA @AAA \\
			f_\xi(y_{\delta,\gamma,0}) @>\subseteq>> f_\zeta(y_{\delta,\gamma,0})
			\end{CD}
			\]

			\item[(d)] if $\xi$ is a limit ordinal, then $f_\xi = \bigcup_{\zeta < \xi} f_\zeta$,
			
			\item[(e)] if $\xi \in S$ and $f_\zeta, f_\zeta^+$ for $\zeta < \xi$ are defined, then
			\begin{enumerate}
				\item[$(e_1)$] $f_\xi = \bigcup_{\zeta < \xi} f_\zeta$ (by clause (d)),
				\item[$(e_2)$] $f_\xi^+ \supseteq \bigcup_{\zeta < \xi} f^+_\zeta$,
				\item[$(e_3)$] $f^+_{\xi}$ is chosen as in Claim \ref{non}, i.e., satisfying:
				\begin{quote}
					$(\ast)_\xi$: For every bad map $g \in \operatorname{End}(G_{f_\xi},+)$,  there is some $\gamma < \lambda$ such that $g$ does not extend to any $\bar{g} \in \operatorname{End}(G_f,+)$ for any $f \in \mathbf{P}_\lambda$ with $f \supseteq f_\xi$ and $f(y_{\xi,\gamma,0}) = f_{\xi}^+(y_{\xi,\gamma,0})$.
				\end{quote}
			\end{enumerate}
			
			\item[(f)] if $\xi \notin S$, we simply let $f^+_{\xi} \in \mathbf{P}_{\xi+1}^+$ be any function extending $f_\xi \cup \bigcup_{\zeta < \xi} f^+_\zeta$,
			
			\item[(g)] $f_{\xi+1} = f_\xi \cup \big( f_\xi^+ \restriction \{y_{\delta,\gamma,0}: (\delta,\gamma) \in (S \cap (\xi+1)) \times (\xi+1)\} \big)$.
		\end{enumerate}
		
		Finally, set $f_\ast = \bigcup_{\xi < \lambda} f_\xi^+$ and   $M:=G_{f_\ast}$ is, the  submodule of $\widehat X$ generated by
		\(
		\bigcup \{ \Gamma^{\xi,\gamma}_{f_\ast(y_{\xi,\gamma,0})} : \xi \in S,\ \gamma < \lambda \}.
		\)
	
	\end{construction}		
	
  By
	Claim~\ref{freeness}, $M$ is $\mu^+$-free as an $R$-module and has the
	stated additive freeness.  Since $X\subseteq M$ and $|X|=\lambda$, while
	all local families have cardinality at most $\lambda$, we also have
	$|M|=\lambda$.
	
	\medskip
	\noindent\textit{Step 6: the final club argument.}
	
	We first record why it is enough to test an endomorphism on the free core.
	The relations $(\sharp)$ imply that, for every positive integer $m$ and every
	local generator $y_n$, some finite $R$-linear combination $x\in X$ satisfies
	\(
	y_n-x\in mM,
	\)
	because a sufficiently long product of the integers $k_j=j+2$ is divisible
	by $m$.  Hence $X$ is dense in $M$ for the intrinsic $\mathbb Z$-adic
	topology.
	Furthermore this topology on $M$ is separated.  Indeed, if
	$a\in\bigcap_{m\geq1}mM$, choose $a_m\in M$ with $ma_m=a$.  The countable
	set $\{a\}\cup\{a_m:m\geq1\}$ is contained, by Claim~\ref{freeness}, in a
	free $R$-submodule $P$ of $M$.  Then
	$a\in\bigcap_{m\geq1}mP$.  As $(R,+)$ is cotorsion-free, it is reduced, and
	so is the direct sum $(P,+)$; therefore this intersection is $0$ and
	$a=0$. By definition, $M$ is separated, as claimed.

	Now let $h\in\End(M,+)$.  If $h\restriction X=\mu_r$ for some $r\in R$ (see Definition~\ref{mulm} for $\mu_r$), then
	$k=h-\mu_r$ vanishes on the dense subgroup $X$.  For $a\in M$ and every
	$m\geq1$, choose $x_m\in X$ with $a-x_m\in mM$.  Then
	$k(a)=k(a-x_m)\in mM$ for all $m$, hence $k(a)=0$ by separatedness.  Thus
	$h=\mu_r$ on all of $M$.
	
	Suppose therefore that $h$ is not of multiplication type.  By the core test
	from Step~1, choose $b\in Z$ such that
	\(
	h(b)\notin Rb.
	\)
	Let
	\[
	\mathcal{D}:=\big\{\xi<\lambda:h[G_{f_\xi}]\subseteq G_{f_\xi}\big\}.
	\]
	The set $\mathcal{D}$ contains a club.  Closedness follows from the continuity property in Construction~\ref{main-construction}(d).  For unboundedness, starting above any $\alpha<\lambda$, alternately
	enlarge the stage so that it contains the image under $h$ of the preceding
	stage and take the supremum after $\omega$ steps; regularity of $\lambda$
	keeps the resulting ordinal below $\lambda$.
	Now choose
	\(
	\xi\in \mathcal{D}\cap S_b
	\)
	large enough that $b,h(b)\in G_{f_\xi}$.  Put
	$g=h\restriction G_{f_\xi}$.  Then $g(b)\notin Rb$, so property $(*)_\xi$
	of Construction~\ref{main-construction} says that $g$ is permanently killed
	by one of the row-$\xi$ local choices.  That choice belongs to the final
	$f_*$, whereas $h$ itself is an endomorphism of $M=G_{f_*}$ extending $g$,
	a contradiction.  
	
	Therefore every additive endomorphism of $M$ is
	multiplication by a scalar from $R$.
	Finally, the assignment \(r\mapsto \mu_r,
	\) defines
	\(
	R\to\End_{\mathbb Z}(M)\),
which is injective because $X$ contains a free basis element: if $\mu_r=0$, then
	$rx=0$ for a free basis element $x$, hence $r=0$.  It is surjective by the
	preceding paragraph.  Thus $\End_{\mathbb Z}(M)\cong R$.
\end{proof}
\iffalse
\begin{remark}\label{rem:cfomega1-correction}
	The original draft stated Theorem~\ref{main2} also for
	$\cf(\mu)=\omega_1$.  The countable local test above does not establish that
	case.  When $\otp(C^\delta_\gamma)=\omega_1$, a countable test sequence is
	bounded in the ladder, and after deleting a bounded initial part the local
	quotient needed in Claim~\ref{freeness} need no longer be free.  To extend
	the present Section~4 proof to cofinality $\omega_1$ one needs an
	$\omega_1$-long local test pair with free quotients at every proper initial
	stage and the same persistent non-extension property.  No such additional
	lemma is proved here.  The statement of Theorem~\ref{main2} is therefore
	restricted to the case actually established by the completion/step-lemma
	argument.
\end{remark}
\fi

\end{document}